\documentclass[11pt,letterpaper,reqno]{amsart}

\usepackage[margin={1.25in, 1.5in}]{geometry}
\usepackage[bottom]{footmisc}

\usepackage[dvipsnames]{xcolor}
\def\TEXTCOLOR {black}
\def\PAGECOLOR {white}
\def\NOTECOLOR {Purple}
\definecolor{linkBlue}{rgb}{.05,.25,.6}

\color{\TEXTCOLOR}
\pagecolor{\PAGECOLOR}

\usepackage{amsmath}
\usepackage{amssymb}
\usepackage{bbm}
\usepackage{mathtools}
\usepackage{amsfonts}
\usepackage{comment}
\usepackage{cancel}
\usepackage{enumitem}
\usepackage[mathscr]{euscript}
\usepackage{tikz-cd}
\usepackage[margin=1in]{caption}
\usepackage{mathrsfs}
\usepackage{appendix}
\usepackage{bibentry}
\usepackage{thmtools}
\usepackage{graphicx}
\graphicspath{{Images/}}
\usepackage{pgfplots}
\pgfplotsset{compat=1.18}
\usepackage[normalem]{ulem}
\usepackage[pdftex, psdextra]{hyperref}
\hypersetup{
	colorlinks=true,
	linkcolor=linkBlue,
	filecolor=magenta,
	urlcolor=linkBlue,
	citecolor=linkBlue}
\usepackage[nameinlink,noabbrev]{cleveref}

\theoremstyle{definition}
\newtheorem{definition}{Definition}[subsection]
\newtheorem*{definition*}{Definition}
\newtheorem{theorem}[definition]{Theorem}
\newtheorem*{theorem*}{Theorem}

\newtheorem{lemma}[definition]{Lemma}
\newtheorem*{lemma*}{Lemma}
\newtheorem{proposition}[definition]{Proposition}
\newtheorem{claim}[definition]{Claim}
\newtheorem{remark}[definition]{Remark}
\newtheorem{example}{Example}[subsection]
\newtheorem*{example*}{Example}
\newtheoremstyle{named}{}{}{\itshape}{}{\bfseries}{.}{.5em}{\thmnote{#3}}
\theoremstyle{named}
\newtheorem*{namedtheorem}{Theorem}

\let\tilde\widetilde
\renewcommand{\bar}[1]{\overline{\hspace{-1pt}{#1}\hspace{-.2pt}}}

\newcommand{\de}{\text{d}}

\newcommand{\R}{\mathbb{R}}

\let\C\relax
\newcommand{\C}{\mathbb{C}}

\renewcommand{\P}{\mathbb{P}}
\renewcommand{\k}{\mathbbm k}

\newcommand{\from}{\,:\,}
\newcommand{\into}{\hookrightarrow}

\newcommand{\acts}{\mathbin{{\text{\raisebox{.75em}{\rotatebox{-90}{$\circlearrowright$}}}}}}

\newcommand{\im}{\text{im}\,}
\newcommand{\sm}{\smallsetminus}

\newcommand{\CP}{\ensuremath{\mathbb {CP}}}

\DeclareMathOperator{\delbar}{\bar\partial}

\DeclareMathOperator{\Id}{Id}

\let\Im\relax

\DeclareMathOperator{\Im}{Im}
\DeclareMathOperator{\tr}{tr}
\DeclareMathOperator{\rk}{rk}

\DeclareMathOperator{\GL}{GL}
\DeclareMathOperator{\SL}{SL}

\DeclareMathOperator{\SU}{SU}
\DeclareMathOperator{\SO}{SO}
\DeclareMathOperator{\Hom}{Hom}
\DeclareMathOperator{\End}{End}
\DeclareMathOperator{\Rep}{Rep}

\DeclareMathOperator{\res}{res}
\DeclareMathOperator*{\twoslash}{\hspace{-.2em}/\hspace{-.2em}/\hspace{-.2em}}
\DeclareMathOperator*{\fourslash}{\twoslash\twoslash}
\newcommand{\bx}{\mathbf{x}}
\newcommand{\by}{\mathbf{y}}

\DeclareFontFamily{U}{wncy}{}
    \DeclareFontShape{U}{wncy}{m}{n}{<->wncyr10}{}
    \DeclareSymbolFont{mcy}{U}{wncy}{m}{n}
    \DeclareMathSymbol{\Sh}{\mathord}{mcy}{"58} 
\usetikzlibrary{decorations.pathreplacing}

\title[Symplectic Structures on Quiver Varieties and  Higgs Bundle Moduli Spaces]{Star-Shaped Nakajima Quiver Varieties, Parabolic Higgs Bundle Moduli Spaces, and their Holomorphic Symplectic Structures}

\author{Arya Yae}
\address{Department of Mathematics\\
University of Oregon, Eugene, OR 97403 USA}
\email{ayae@uoregon.edu}

\date{} 
\begin{document}

\begin{abstract}
    In this paper, we consider two classes of hyperk\"ahler manifolds: moduli spaces of central-Levi parabolic Higgs bundles on the punctured sphere and star-shaped Nakajima quiver varieties.  We produce a map $\mathcal T$ from a given star-shaped quiver variety $\mathcal X$ to a central-Levi parabolic Higgs bundle moduli space $\mathcal M$.  We verify that $\mathcal T$ preserves stability and we show that it is a homeomorphism onto the locus of Higgs bundles with trivial underlying holomorphic structure.  We then prove our main theorem: that $\mathcal T$ identifies the natural holomorphic symplectic structures on the two spaces.  This theorem generalizes work by Biswas, Florentino, Godinho, Mandini from the rank 2, full flag, strongly parabolic case to arbitrary rank, partial flag, and weakly parabolic cases---namely, those whose Higgs field residues project to the centers of their respective Levi subalgebras.
\end{abstract}

\maketitle

\tableofcontents

\newcommand{\app}{\text{app}}

\section{Introduction}
Nakajima quiver varieties are hyperk\"ahler manifolds introduced by Nakajima in \cite{Nak94} and studied further in \cite{KN90} in the 4-dimensional cases, where they are shown to model every hyperk\"ahler ALE space in Kronheimer's classification.  For $\Gamma=D_4,E_6,E_7,E_8$, the quivers used to model ALE-$\Gamma$ spaces are the ``star-shaped'' ones as depicted in \Cref{fig: quiver} with a central node and no cycles or double arrows.  On the other hand, the moduli spaces of parabolic $\SL(r,\C)$-Higgs bundles were introduced by Simpson \cite{Sim90}, generalizing the moduli spaces of Higgs bundles on unpunctured Riemann surfaces introduced by Hitchin in \cite{Hit87}.  As in the unpunctured case, under suitable conditions the parabolic Higgs bundle moduli spaces are smooth hyperk\"ahler manifolds \cite{Kon93,BiquardBoalch,CFW24}.

Our goal is as follows: given a moduli space $\mathcal M=\mathcal M(\sigma,\alpha)$ of parabolic $\SL(r,\C)$-Higgs bundles on $\CP^1$ with suitable parabolic weights, we produce a star-shaped Nakajima quiver variety $\mathcal X=\mathcal X(\tau,\beta)$ and a map $\mathcal T\from\mathcal X\to\mathcal{M}$.  This map was proposed by Rayan and Schaposnik \cite{RS21} in the strongly parabolic ($\sigma=0$) case, with further restrictions on the parabolic weights.\footnote{
	Namely $\alpha_i^{(k)}=0$ for $k\geq2$. Note that they do not verify that the map preserves stability.
}  We then prove
\begin{namedtheorem}[\Cref{thm: T preserves stability}]
	For suitable parabolic weights $\alpha$ and GIT weights $\beta$, the map $\mathcal T$ takes $\beta$-stable quiver representations to $\alpha$-stable Higgs bundles.  As $\mathcal T$ also takes equivalent quiver representations to gauge-equivalent Higgs bundles, it is well-defined.
\end{namedtheorem}
We show in \Cref{thm: T is a homeomorphism onto its image} that $\mathcal T$ is a homeomorphism onto the locus $\mathcal M_0\subseteq\mathcal M$ consisting of Higgs bundles with trivial underlying holomorphic structure.  This leads to our main theorem:
\begin{namedtheorem}[\Cref{thm: holo symplectic forms agree}]
	Under the embedding $\mathcal T\from\mathcal X\into\mathcal M$, the holomorphic symplectic form $\Omega_\mathrm{Higgs}$ pulls back to $2\pi\Omega_\mathrm{quiv}$, where $\Omega_\mathrm{quiv}$ is the holomorphic symplectic form on $\mathcal X$.
\end{namedtheorem}

\subsection*{Relation to other work}
In case of rank 2 strongly parabolic Higgs bundles with full flags, the corresponding quiver variety $\mathcal X$ is the so-called ``hyperpolygon'' space studied by Konno in \cite{Kon02}.  In this case, Godinho and Mandini \cite{GM11} describe the map $\mathcal T$ and prove the analogue of \Cref{thm: T preserves stability}.  In \cite{BFGM15}, Biswas, Florentino, Godinho, and Mandini prove that $\mathcal T$ preserves the holomorphic symplectic structure.  In \cite{FY26}, the author and Fredrickson describe a certain degenerate limit under which the full hyperk\"ahler structure on $\mathcal M$ converges to the hyperk\"ahler structure on $\mathcal X$ (the underlying holomorphic symplectic structure remains fixed throughout this degeneration).  This hyperk\"ahler metric degeneration was independently proved by Heller, Heller, and Meneses in \cite{HHM25} using different techniques.

The present work serves to generalize the results of \cite{GM11} and \cite{BFGM15} in three ways: (1) generalize to higher ranks, (2) generalize to parabolic structures with partial flags, and (3) remove the strongly parabolic condition (which required $\res_p\varphi$ to be strictly block upper triangular with respect to the flag).  This paper is a necessary step toward generalizing \cite{FY26}, which shall be the focus of upcoming work.

This project fits into a larger story about the so-called ``gravitational instantons'' of type ALE and ALG and their higher dimensional analogues.  It is conjectured \cite{aim, CherkisVideo} that all ALG gravitational instantons have constructions as moduli spaces of Higgs bundles.  In the rank 2 case with four marked points on $\CP^1$, the parabolic Hitchin moduli spaces are four-dimensional---in the strongly parabolic case ($\sigma=0$) they are shown to be ALG-$D_4$ instantons in \cite{FMSW21}, and by extending to the weakly parabolic cases the paper \cite{FMSW26} and a forthcoming paper by the same authors use a Torelli-type theorem to show that every ALG-$D_4$ instanton can be modeled this way.

Our generalization to higher rank with partial flags include quiver varieties and Hitchin moduli spaces which should correspond to the $E_6,E_7,E_8$ type instantons, as well as higher dimensional quiver varieties and parabolic Hitchin moduli spaces---the former\footnote{
	In fact, Dimakis and Rochon prove that \emph{all} Nakajima quiver varieties have quasi-asymptotically conical hyperk\"ahler structures.
} are quasi-asymptotically conical \cite{DR25}, and the latter should be some higher-dimensional analogue of ALG$^{(*)}$ spaces.  Our generalization to include the weakly parabolic case is also significant: for example, in the 12-parameter families of ALG-$D_4$ and ALE-$D_4$ instantons, only 4-dimensional subfamilies can be modeled using strongly parabolic Higgs bundles and quiver varieties with complex moment $\mu_\C=0$, respectively.  Therefore, this paper and its sequel aims to make substantial progress toward a conjecture by Cherkis \cite{CherkisVideo} which states that all of these ALG instantons of the $A_k, D_k, E_6, E_7, E_8$ type should have degenerations of the full hyperk\"ahler structure to that of the corresponding ALE spaces.

In restricting our focus to star-shaped Nakajima quiver varieties, we do not consider quivers with double edges or loops.  Among the excluded quivers are those which should correspond to moduli spaces of wild Higgs bundles on $\CP^1$ as defined in \cite{BiquardBoalch} (doubled edges), and ones which conjecturally correspond to parabolic Hitchin moduli spaces on higher genus surfaces \cite{RS21} (loops on the central node).  As a result, the gravitational instantons of type ALG-$E_k^*$\footnote{
	Cherkis sometimes refers to these as ``sha''-spaces, since the Cyrillic letter `$\Sh$' is a reflection of `E'.
} (those dual to ALG-$E_k$ instantons) are outside the scope of this project.

\subsection*{Outline of Paper}
In \Cref{sec: quiver varieties} we review the construction of star-shaped Nakajima quiver varieties and discuss some stability results.  In \Cref{sec: Higgs bundle moduli space construction} we review the construction of the parabolic Hitchin moduli spaces and describe the holomorphic symplectic structure.  In \Cref{sec: the map} we construct the map $\mathcal T$ from the quiver varieties to the parabolic Hitchin systems, prove it is well-defined by showing stability is preserved, and prove that its image is the subspace consisting of Higgs bundles with trivial underlying holomorphic structure.  Finally, in \Cref{sec: holomorphic symplectic forms} we prove the main theorem---first with a clean proof in the strongly parabolic case by comparing the natural tautological 1-forms, and then with a laborious proof in the general case.

\subsection*{Acknowledgements}
The author would like to thank Laura Fredrickson for many useful discussions, and Nicolas Addington for providing references on King stability.

\section{Star-Shaped Nakajima Quiver Varieties}\label{sec: quiver varieties}

\subsection{Construction}\label{subsec: quiver variety construction}

We follow Nakajima's original construction \cite{Nak94}.
\begin{figure}
	$$
	\scalebox{1.0}{
	\begin{tikzcd}[ampersand replacement=\&, row sep=.3cm, column sep=1cm]
				\&\&r_1^{(\ell_1)}\arrow[lldd]	\&\cdots\arrow[l]	\&r_1^{(2)}\arrow[l]	\&r_1^{(1)}\arrow[l]\\
				\&\&r_2^{(\ell_2)}\arrow[lld]		\&\cdots\arrow[l]	\&r_2^{(2)}\arrow[l]	\&r_2^{(1)}\arrow[l]\\
		r_\star	\&\&\vdots\\
				\&\&r_n^{(\ell_n)}\arrow[llu]		\&\cdots\arrow[l]	\&r_n^{(2)}\arrow[l]	\&r_n^{(1)}\arrow[l]
	\end{tikzcd}
	}$$
	\caption{A star-shaped quiver $Q$ with dimension vector $\vec r$.}
	\label{fig: quiver}
\end{figure}
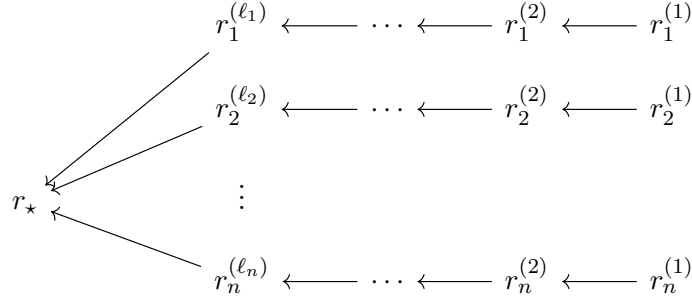
Let $Q$ be the quiver shown in \Cref{fig: quiver} with the distinguished dimension vector $\vec r=\left(r_\star,(r_i^{(k)})_{i,k}\right)$ where $r_i^{(k)}<r_i^{(k+1)}$ for all $i=1,\dots,n$ and $k=1,\dots,\ell_i$ (to simplify expressions we take $r_i^{(\ell_i+1)}:=r_\star$).  We call such a a quiver \emph{star-shaped}, and such a dimension vector \emph{branch-increasing}.  The vector space of representations with dimension vector $\vec r$ is
\begin{equation}
	\Rep_{\vec r} Q=\!\bigoplus_{\substack{i=1,\dots,n\\k=1,\dots,\ell_i}}\!\Hom\left(\C^{r_i^{(k)}},\C^{r_i^{(k+1)}}\right),
\end{equation}
and $X:=T^*\Rep_{\vec r} Q$ is the cotangent bundle.  We denote elements $(\bx,\by)\in X$ where $\bx\in\Rep_{\vec r} Q$ and $\by\in T_\bx^*\Rep_{\vec r} Q$, each of which split into components $x_i^{(k)}\in\Hom\left(\C^{r_i^{(k)}},\C^{r_i^{(k+1)}}\right)$ and $y_i^{(k)}\in\Hom\left(\C^{r_i^{(k+1)}},\C^{r_i^{(k)}}\right)$ for $i=1,\dots,n$ and $k=1,\dots,\ell_i$.  Individually, each $x_i^{(k)}$ is a morphism along the corresponding arrow in \Cref{fig: quiver}, and each $y_i^{(k)}$ is a morphism in the reverse direction.  Thus, we can write $X=\Rep_{\vec r}\tilde Q$ where $\tilde Q$ is the doubled quiver with arrows going both directions.  We can also regard $\bx$ and $\by$ as endomorphisms of $V:=\C^{r_\star}\oplus\bigoplus_{i,k}\C^{r_i^{(k)}}$ supported in blocks according to the arrows in $Q$ and the reversed quiver $\bar Q$, respectively.  We call such a $\bx\in\End(V)$ \emph{$Q$-supported} and such a $\by\in\End(V)$ \emph{$\bar Q$-supported}.

This $X$ has a standard hyperk\"ahler structure $(X,J_1,J_2,J_3,\omega,\Omega,g)$.  In particular, the holomorphic symplectic form $\Omega$ is the exterior derivativer of the tautological 1-form $L$ on $X$ is given by
\begin{equation}\label{eqn: quiver tautological 1 form}
	L_{(\bx,\by)}(\dot\bx,\dot\by)=\tr(\by\dot\bx).
\end{equation}
Note that this description favors a distinguished complex structure $J_1$ and the natural 2-forms $\omega=\omega_{J_1}$ associated to $J_1$, and the holomorphic symplectic form decomposes as $\Omega=\omega_{J_2}+i\omega_{J_3}$.

\begin{figure}
	\[
	\scalebox{1.0}{
	\begin{tikzcd}[ampersand replacement=\&, row sep=.8cm, column sep=1.2cm]
			\&\&r_1^{(\ell_1)}\arrow["x_1^{(\ell_1)}", lld]\arrow["y_1^{(\ell_1-1)}", r, bend left=20]		\&\cdots\arrow["x_1^{(\ell_1-1)}",l]\arrow["y_1^{(2)}", r, bend left=20]	\&r_1^{(2)}\arrow["x_1^{(2)}", l]\arrow["y_1^{(1)}", r, bend left=20]	\&r_1^{(1)}\arrow["x_1^{(1)}", l]\\
		r_\star\arrow["y_1^{(\ell_1)}", rru, bend left=20]\arrow["y_n^{(\ell_n)}", rrd, bend left=20]	\&\&\vdots\\
			\&\&r_n^{(\ell_n)}\arrow["x_n^{(\ell_n)}",llu]\arrow["y_n^{(\ell_n-1)}", r, bend left=20]		\&\cdots\arrow["x_n^{(\ell_n-1)}", l]\arrow["y_n^{(2)}", r, bend left=20]	\&r_n^{(2)}\arrow["x_n^{(2)}", l]\arrow["y_n^{(1)}", r, bend left=20]	\&r_n^{(1)}\arrow["x_n^{(1)}", l]
	\end{tikzcd}
	}
	\]
	\caption{A representation $(\bx,\by)$ of the quiver.}
	\label{fig: a quiver representation}
\end{figure}
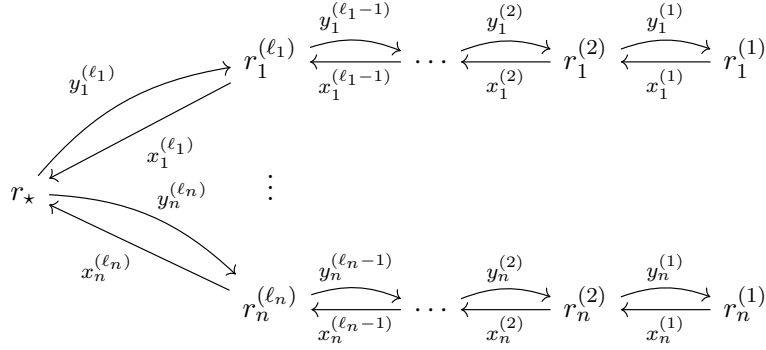

The space $X$ respectively admits $\Omega$- and $\omega$-Hamiltonian actions of the groups of complex and unitary transformations
\begin{align}
	G_\C
		&=\SL(r_\star,\C)\times\!\prod_{\substack{i=1,\dots,n\\k=1,\dots,\ell_i}}\!\GL(r_i^{(k)},\C),\\
	G
		&=\SU(r_\star)\times\!\prod_{\substack{i=1,\dots,n\\k=1,\dots,\ell_i}}\!U(r_i^{(k)}).
\end{align}
Namely, regarding an element $g\in G_\C$ as a \emph{$Q$-block-diagonal} endomorphism of $V$, the action is $(\bx,\by)\mapsto(g\bx g^{-1},g\by g^{-1})$.  Writing the components as $g=\big(g_\star,(g_i^{(k)})_{i,k}\big)$ and taking $g_i^{(\ell_i+1)}$ to be $g_\star$, the action is
	$$x_i^{(k)}\mapsto g_i^{(k+1)}x_i^{(k)}(g_i^{(k)})^{-1},
		\qquad y_i^{(k)}\mapsto g_i^{(k)}y_i^{(k)}(g_i^{(k+1)})^{-1}.$$
The moment maps $\mu_\C\from X\to\mathfrak g_\C^*$ and $\mu_\R\from X\to\mathfrak g^*$ for these actions (see \cite{Nak94}) can be packaged very nicely as $Q$-block diagonal endomorphisms of $V$:
\begin{align}
	\mu_\C(\bx,\by)
		&=\varpi([\by,\bx]),
			\label{eqn: quiver complex moment map}\\
	\mu_\R(\bx,\by)
		&=\varpi\left([\bx^\dagger,\bx]+[\by^\dagger,\by]\right).
			\label{eqn: quiver real moment map}
\end{align}
Here, $\varpi$ is the restriction map\footnote{
	Under the identification $\mathfrak g_\C\cong\mathfrak g_\C^*$ given by the trace pairing $\langle A,B\rangle=\tr(AB)$, this $\varpi$ becomes the orthogonal projection with respect to the Frobenius inner product.
	}
		\[\mathfrak{gl}(r_\star,\C)^*\times\!\prod_{\substack{i=1,\dots,n\\k=1,\dots,\ell_i}}\!\mathfrak{gl}(r_i^{(k)},\C)^*\to\mathfrak{sl}(r_\star,\C)^*\times\!\prod_{\substack{i=1,\dots,n\\k=1,\dots,\ell_i}}\!\mathfrak{gl}(r_i^{(k)},\C)^*\]
and the Lie bracket is the standard one on $\End(V)$.  The moment maps decompose into $\mu_\C=\mu_{\C,\star}+\sum_{i,k}\mu_{\C,i}^{(k)}$ and $\mu_\R=\mu_{\R,\star}+\sum_{i,k}\mu_{\R,i}^{(k)}$ where
\begin{align}
	\mu_{\C,\star}(\bx,\by)
		&=\sum_i-(x_i^{(\ell_i)}y_i^{(\ell_i)})_0\in\mathfrak{sl}(r_\star,\C),
			\label{eqn: mu SL}\\
	\mu_{\C,i}^{(k)}(\bx,\by)
		&=y_i^{(k)}x_i^{(k)}-x_i^{(k-1)}y_i^{(k-1)}\in\mathfrak{gl}(r_i^{(k)},\C),
			\label{eqn: mu C i}\\
	\mu_{\R,\star}(\bx,\by)
		&=\frac{\mathbbm i}2\left(-x_i^{(k)}(x_i^{(k)})^\dagger+(y_i^{(k)})^\dagger y_i^{(k)}\right)_0\in\mathfrak{su}(r_\star),
			\label{eqn: mu SU}\\
	\mu_{\R,i}^{(k)}(\bx,\by)
		&=\frac{\mathbbm i}2\left((x_i^{(k)})^\dagger x_i^{(k)}-y_i^{(k)}(y_i^{(k)})^\dagger-x_i^{(k-1)}(x_i^{(k-1)})^\dagger+(y_i^{(k-1)})^\dagger y_i^{(k-1)}\right)\in\mathfrak u(r_i^{(k)}).
			\label{eqn: mu R j}
\end{align}
where the subscript 0 means we take the trace-free part, i.e.\ $A_0=\frac{\tr A}{r^\star}\mathrm{Id}$, and we identify the Lie algebras with their duals via the trace pairing.

We make the natural identification of the centers of the Lie algebras $Z(\mathfrak g_\C)\cong\bigoplus_i\C^{\ell_i}$ and $Z(\mathfrak g)\cong\bigoplus_i\R^{\ell_i}$.  For $\tau=(\tau_i^{(k)})_{i,j}\in Z(\mathfrak g_\C)$ and $\beta=(\beta_i^{(k)})_{i,k}\in\bigoplus_i\R_{>0}^{\ell_i}\subset Z(\mathfrak g)$, we define the hyperkähler quotient
	$$\mathcal X(\tau,\beta)=X\fourslash_{\tau,\beta}G=\left(\mu_\C^{-1}(\tau)\cap\mu_\R^{-1}(\beta)\right)/G.$$
The full hyperk\"ahler structure descends to the quotient, but we remark that the tautological 1-form $L$ on $X$ only descends to a canonical holomorphic 1-form on $\mathcal X(\tau,\beta)$ if $\tau=0$.

\begin{example}[$n$-sided hyperpolygon space]
	In the case where $(Q,\vec r)$ has $r_\star=2$ and full flags, $\mathcal X_{(Q,\vec r)}(0,\beta)$ is the $n$-sided hyperpolygon space introduced in \cite{Kon02} and studied further in \cite{HP04}.  As described in \cite[Remark 2.1.6]{FY26}, $\mathcal X_{(Q,\vec r)}(0,\beta)$ can be viewed as the moduli space of ``telescoping polygons''\footnote{
		Each side of the telescoping polygon is also equipped with a phase, and these phases are subject to additional constraints.
	} in $\R^3$, up to the action of $\SO(3)$.
	\begin{figure}[h]
	\centering
	\resizebox{4cm}{!}{
	\begin{tikzpicture}
		\tikzstyle{every node}=[font=\LARGE]
		\draw [ fill=black , line width=1pt ] (16,2) circle (.2cm);
		\draw [line width=4pt] (16,2) -- (10,5);
			\draw [decorate, decoration = {brace,amplitude=12pt,raise=12pt}] (16,2) -- (10,5)
				node[midway,xshift=-6ex,yshift=-4em,font=\Huge]{\scalebox{1.2}{$\sqrt2\beta_1$}};
		\draw [line width=4pt, ->, >=Stealth, dashed] (10,5) -- (5,7.5);
			\node[font=\Huge, xshift=0em, yshift=3.5em] at (10.5,4.75) {\scalebox{1.2}{$v_1$}};
		\draw [ color=green!40!black, line width=4pt, ->, >=Stealth, dashed] (0,10) -- (5,7.5);
			\node[font=\Huge, xshift=2em, yshift=2em, color=green!40!black] at (2.5,8.75) {\scalebox{1.2}{$w_1$}};

		\draw [ fill=black , line width=1pt ] (0,10) circle (.2cm);
		\draw [line width=4pt] (0,10) -- (3,16);
			\draw [decorate, decoration = {brace,amplitude=12pt,raise=12pt}] (0,10) -- (3,16)
				node[midway,xshift=-15ex,yshift=2.5em,font=\Huge]{\scalebox{1.2}{$\sqrt2\beta_2$}};
		\draw [line width=4pt, ->, >=Stealth, dashed] (3,16) -- (5,20);
			\node[font=\Huge, xshift=3.5em, yshift=0em] at (2.5,15) {\scalebox{1.2}{$v_2$}};
		\draw [ color=green!40!black, line width=4pt, ->, >=Stealth, dashed] (7,24) -- (5,20);
			\node[font=\Huge, xshift=2em, yshift=-3em, color=green!40!black] at (6,22) {\scalebox{1.2}{$w_2$}};

		\draw [ fill=black , line width=1pt ] (7,24) circle (.2cm);
		\draw [line width=4pt] (7,24) -- (10.5,24);
			\draw [decorate, decoration = {brace,amplitude=12pt,raise=12pt}] (7,24) -- (10.5,24)
				node[midway,xshift=0ex,yshift=5em,font=\Huge]{\scalebox{1.2}{$\sqrt2\beta_3$}};
		\draw [line width=4pt, ->, >=Stealth, dashed] (10.5,24) -- (12.5,24);
			\node[font=\Huge, xshift=3em, yshift=-2.5em] at (9.25,24) {\scalebox{1.2}{$v_3$}};
		\draw [ color=green!40!black, line width=4pt, ->, >=Stealth, dashed] (14.5,24) -- (12.5,24);
			\node[font=\Huge, xshift=0em, yshift=-2.5em, color=green!40!black] at (13.5,24) {\scalebox{1.2}{$w_3$}};

		\draw [ fill=black , line width=1pt ] (14.5,24) circle (.2cm);
		\draw [line width=4pt] (14.5,24) -- (15, 24-22/3);
			\draw [decorate, decoration = {brace,amplitude=12pt,raise=12pt}] (14.5,24) -- (15, 24-22/3)
				node[midway,xshift=14ex,yshift=.2em,font=\Huge]{\scalebox{1.2}{$\sqrt2\beta_4$}};
		\draw [line width=4pt, ->, >=Stealth, dashed] (15, 24-22/3) -- (15.5, 24-2*22/3);
			\node[font=\Huge, xshift=-3em, yshift=-1em] at (15, 24-22/3) {\scalebox{1.2}{$v_4$}};
		\draw [ color=green!40!black, line width=4pt, ->, >=Stealth, dashed] (16,2) -- (15.5, 24-2*22/3);
			\node[font=\Huge, xshift=-3em, yshift=4em, color=green!40!black] at (15.7, 6) {\scalebox{1.2}{$w_4$}};

		\draw [line width=4pt, loosely dotted] (16,2) -- (7,24);x
		\draw [line width=4pt, blue] (11.5,15.5) arc (65:160:1.5);
	\end{tikzpicture}
	}
	\caption{A hyperpolygon shown as vectors in $\mathfrak{su}(2)\cong\R^3$}
	\label{fig: hyperpolygon}
	\end{figure}
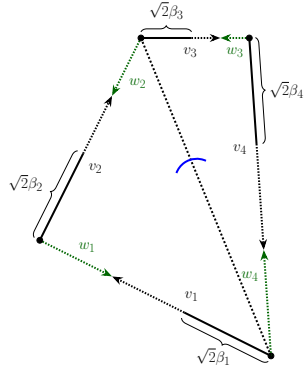
\end{example}

The higher rank geometric descriptions of $\mu_\C^{-1}(\tau)\cap\mu_\R^{-1}(\beta)$ are much more complicated.  Although the moment map conditions are only quadratic equations, it seems unlikely that one could use these descriptions to explicitly describe the hyperk\"ahler metric on $\mathcal X$ even in the case of hyperpolygon space.

\subsection{Stability}
By the Kempf-Ness theorem, we have the identification
	$$\mathcal X(\tau,\beta)=\mu_\C^{-1}(\tau)^{\chi_\beta-\text{st}}\twoslash G_\C$$
as holomorphic symplectic manifolds for the following notion of $\chi_\beta$-stability.

\begin{definition}[\cite{King94}]
	Let $\chi_\beta\from G\to\C$ be the character given by
		$$\chi_\beta(g)=\!\prod_{\substack{i=1,\dots,n\\k=1,\dots,\ell_i}}\!\det(g_i^{(k)})^{\beta_i^{(k)}}.$$
	Let $L\to X$ be a trivial rank 1 bundle, and let $G\acts L$ by $g\cdot((\bx,\by),l)=(g\cdot(\bx,\by),\chi(g)\,l)$.  Then $(\bx,\by)\in\mu_\C^{-1}(\tau)$ is \emph{$\chi_\beta$-semistable} if the closure of the orbit of $((\bx,\by),1)$ in $L$ does not intersect the image $Z$ of the zero section.  If in addition the stabilizer of $(\bx,\by)$ is finite, then $(\bx,\by)$ is \emph{$\chi_\beta$-stable}.
\end{definition}

To match the setting of \cite{King94}, we extend the $G_\C$-action to the larger Lie group $\tilde G_\C=\GL(r_\star,\C)\!\times\!\prod_{\substack{i=1,\dots,n\\k=1,\dots,\ell_i}}\!\GL(r_i^{(k)},\C)$, and extend the character as follows.
\begin{definition}[\cite{King94}]\label{def: extended GIT weight vector}
	We extend $\chi_\beta$ to a character $\chi_{\tilde\beta}$ on $\tilde G_\C$ by picking the weight
	\begin{equation}\label{eqn: beta star}
		\beta_\star=-\frac1{r_\star}\sum_{\substack{i=1,\dots,n\\k=1,\dots,\ell_i}}r_i^{(k)}\beta_i^{(k)}
	\end{equation}
	on the central node $\star$.  We call the tuple $\tilde\beta=\left(\beta_\star,(\beta_i^{(k)})_{i,k}\right)$ the \emph{extended GIT weight vector}.  Given a $\tilde Q$-representation $S$ with dimension vector $\vec s=\left(s_\star,(s_i^{(k)})_{i,k}\right)$, the $\tilde\beta$-weight of $S$ is
		$$\tilde\beta(S)=\vec s\cdot\tilde\beta=s_\star\beta_\star+\sum_{i,k}s_{S,i}^{(k)}\beta_i^{(k)}.$$
\end{definition}
Because of our choice of $\beta_\star$, since $(\bx,\by)$ has dimension vector $\vec r$ we have $\tilde\beta((\bx,\by))=0$.

\begin{definition}[\cite{King94}]
	A representation $R$ of $Q$ is $\tilde\beta$-semistable (resp.\ stable) if both of the following hold:
	\begin{itemize}
		\item $\tilde\beta(R)=0$;
		\item for all nonzero proper subrepresentations $S\subset R$, we have $\tilde\beta(S)\leq\tilde\beta(R)$ (resp.\ $\tilde\beta(S)<\tilde\beta(R))$.
	\end{itemize}
\end{definition}

\begin{proposition}[\cite{King94}]\label{prop: King stability}
	A representation $(\bx,\by)\in\Rep_{\vec r}Q$ is $\chi_\beta$-semistable (resp.\ stable) if and only if it is $\beta$-semistable (resp.\ stable).
\end{proposition}

Because the two notions of stability agree with each other, we will simply refer to representations as being ``$\beta$-(semi)stable.''

The following lemma gives an important consequence of $\beta$ stability.
\begin{lemma}\label{lemma: the x are injective}
	If $(\bx,\by)\in\mu_\C^{-1}(\tau)$ is $\beta$-semistable, then each $x_i^{(k)}$ is injective.
\end{lemma}
\begin{proof}
	Suppose some $x_i^{(k_0)}$ is not injective, say $K=\ker x_i^{(k_0)}$.  We construct a subrepresentation $S$ as follows.  We choose the vector space $S_i^{(k_0)}=K$, and for $k=1,\dots,k_0-1$ we set $S_i^{(k)}=\left(x_i^{(k)}\cdots x_i^{(k_0-1)}\right)^{-1}(K)$.  At all remaining nodes we choose the zero vector spaces (see \Cref{fig: subrepresentation S}).  For the homomorphisms, we just restrict the various $x_i^{(k)}$ and $y_i^{(k)}$.  It is easy to verify that $S$ is a well-defined subrepresentation using the condition $\mu_\C(\bx,\by)=\sigma$ (the proof is analogous to proof of \Cref{claim: subrep well defined} below).  We have $\tilde\beta(S)>0$ since all $\beta_i^{(k)}>0$, so $(\bx,\by)$ is unstable.  This proves the lemma.

	\begin{figure}
	$$
	\begin{tikzcd}[row sep=0.3cm, column sep=1.2cm]
				&&0\arrow[lldd]\arrow[r, bend left=20]	&\cdots\arrow[l]\arrow[r, bend left=20]	&0\arrow[l]\\
				&&\vdots\\
		0\arrow[rruu, bend left=20]\arrow[rr, bend left=20]\arrow[rrdd, bend left=20]		&&\cdots\arrow[ll]\arrow[r, bend left=20]		&0\arrow[l]\arrow[r, bend left=20]	&S_i^{(k_0)}\arrow[l]\arrow[r, bend left=20]	&\cdots\arrow[l]\arrow[r, bend left=20]	&S_i^{(1)}\arrow[l]\\
				&&\vdots\\
				&&0\arrow[lluu]\arrow[r, bend left=20]		&\cdots\arrow[l]\arrow[r, bend left=20]	&0\arrow[l]
	\end{tikzcd}
	$$
	\caption{The subrepresentation $S$}
	\label{fig: subrepresentation S}
	\end{figure}
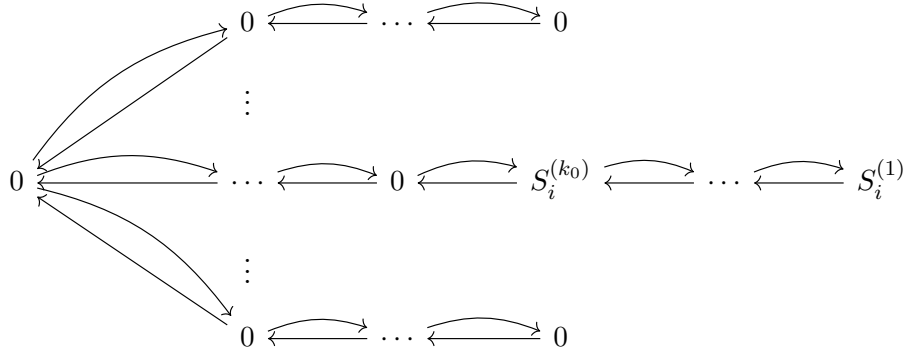
\end{proof}

\subsection{\texorpdfstring{Holomorphic Symplectic Form on $\mathcal X$}{Holomorphic Symplectic Form on X}}
We begin with a definition:
\begin{definition}\label{def: unitary deformation of quiver rep}
	Let $(\bx,\by)\in\mu_\C^{-1}(\tau)\cap\mu_\R^{-1}(\beta)$.  A tangent vector $(\dot\bx,\dot\by)\in T_{(\bx,\by)}X$ is a \emph{unitary deformation} if it is orthogonal to the image of the linearized action map $\de\rho\from\mathfrak g_\C\to T_{(\bx,\by)}\mu_\C^{-1}(\tau)$.
\end{definition}
Recall from \eqref{eqn: quiver tautological 1 form} the tautological (holomorphic) 1-form $L$ on $X=T^*\Rep_{\vec r}Q$ given by
\begin{equation}\label{eqn: quiver taut 1 form expanded}
	L\big|_{(\bx,\by)}(\dot\bx,\dot\by)=\tr(\dot\bx\by)=\sum_{i,k}\tr(\dot x_i^{(k)}y_i^{(k)}).
\end{equation}
Taking the exterior derivative, we recover the holomorphic 2-form $\Omega=\de L$.  Evaluated on a pair of of deformation $\dot\bx_j=(\dot x_{j,i}^{(k)})_{i,k}$, $\dot\by_j=(\dot y_{j,i}^{(k)})_{i,k}$ for $j=1,2$,
\begin{equation}\label{eqn: quiver holo symp form}
	\Omega\big|_{(\bx,\by)}((\dot\bx_1,\dot\by_1),(\dot\bx_2,\dot\by_2))=\tr(\dot\bx_1\dot\by_2-\dot\bx_2\dot\by_1)=\sum_{i,k}\tr\left(\dot x_{1,i}^{(k)}\dot y_{2,i}^{(k)}-\dot x_{2,i}^{(k)}\dot y_{1,i}^{(k)}\right).
\end{equation}
The latter descends to a holomorphic symplectic form $\Omega_\mathrm{quiv}$ on $\mathcal X=\mathcal X_{(Q,\vec r)}(\tau,\beta)$ given by the same expression, where one uses the identification of $T_{(\bx,\by)}\mathcal X$ with the subspace of unitary deformations of $(\bx,\by)$.\footnote{
	Since the formulas for $L_\mathrm{quiv}$ and $\Omega_\mathrm{quiv}$ are invariant under complex gauge transformations, one can actually relax the gauge orthogonality condition and only assume $(\dot\bx,\dot\by)\in\ker\de\mu_\C\big|_{(\bx,\by)}$.  For our purposes, however, we may as well assume the deformations are unitary.
} of $(\bx,\by)$ when applying the above formula for $\Omega_\mathrm{quiv}$.  In the case $\tau=0$, the tautological 1-form also descends and is calculated in the same way.  See \cite[Section 2.1.2]{FY26} for more details.

\section{Parabolic Higgs Bundle Moduli Space Analytic Construction}\label{sec: Higgs bundle moduli space construction}

To build the moduli space of parabolic $\SL(r,\C)$-Higgs bundles, one fixes the following data:
\begin{itemize}
	\item A divisor $D$ on a complex curve $C$;
	\item A vector bundle $E\to C$;
	\item Filtrations of the sheaves of germs of sections of $E$ near points in $D$, which coincide with filtrations of the fibers $E_p$ for $p\in D$ along with parabolic weights $\alpha_i^{(k)}\in(-\frac12,\frac12)$;
	\item Prescribed Higgs field residue data at each point in $D$, namely an element $\sigma$ of the center of the Lie algebra of endomorphisms of the associated graded rings of $E_p$.
\end{itemize}
For the present work, we specialize to $C=\CP^1$ and take $E$ to be the trivial complex bundle of rank $r$.

\subsection{Moduli Space Construction}
Fix a divisor $D=\{p_1,\dots,p_n\}$ of $n$ distinct points in $\CP^1$.

\begin{definition}[Parabolic Bundle of type $(Q,\vec r)$]\label{def: parabolic bundle}
	Let $Q$ be a star-shaped quiver with $n$ branches and let $\vec r=(r_\star,(r_i^{(k)})_{i,k})$ be a branch-increasing dimension vector as in \Cref{fig: quiver}.  Let $E\to C=\CP^1$ be a trivial rank $r_\star$ vector bundle.  A \emph{parabolic structure of type $Q$} on $E$ is the data of filtrations $F_i^\bullet$ of $E_{p_i}$ along with parabolic weights $\alpha_i^\bullet$:
		\[
		\begin{array}[column sep = -5pt]{ccccccccc}
			E_{p_i}		&=	&F_i^{(\ell_i+1)}       &\supset    &\cdots &\supset    &F_i^{(1)}      &\supset          &0\\[4pt]
			&		&\alpha_i^{(\ell_i+1)}  &<          &\cdots &<          &\alpha_i^{(1)} &&
		\end{array}
		\]
	where $\dim F_i^{(k)}=r_i^{(k)}$ and $\alpha_i^{(k)}\in(-\frac12,\frac12)$ for all $i=1,\dots,n$ and $k=1,\dots,\ell_i+1$.  We use $\mathcal E$ to denote the data $(E, \mathcal F, \alpha)$.  The \emph{multiplicity} of $\alpha_i^{(k)}$ is the dimension $m_i^{(k)}$ of the associated graded piece $(\mathrm{gr}F_i)^{(k)}:=F_i^{(k)}/F_i^{(k-1)}$ (we take $F_i^{(0)}=0$).
\end{definition}

\begin{definition}
    The \textit{parabolic degree} of a parabolic bundle $\mathcal E$ with a holomorphic structure $\delbar_E$ is $\text{pardeg}\,\mathcal E=\deg(E,\delbar_E)+\sum_{i,k}m_i^{(k)}\alpha_i^{(k)}$.  The \textit{slope} of $\mathcal E$ is $\text{slope}\,\mathcal E=\frac{\text{pardeg}\,\mathcal E}{\rk\mathcal E}$.
\end{definition}

\begin{definition}
	Let $\mathcal E=(E,\mathcal F_\mathcal E,\alpha_\mathcal E)$ and $\mathcal H=(H,\mathcal F_\mathcal H,\alpha_\mathcal H)$ be two parabolic bundles of type $(Q_1,\vec r_1)$ and $(Q_2,\vec r_2)$ respectively, each of these being star-shaped quivers with $n$ branches.  A \emph{parabolic homomorphism} $L\in\Hom_\mathrm{par}(\mathcal E,\mathcal H)$ is a map of vector bundles $E\to H$ which sends each $F_{\mathcal E,i}^{(k)}$ into some $F_{\mathcal H,i}^{(j)}$ with $\alpha_{\mathcal E,i}^{(k)}\leq\alpha_{\mathcal H,i}^{(j)}$.  Such a map is a \emph{strongly parabolic homomorphism} if it sends each $F_{\mathcal E,i}^{(k)}$ into some $F_{\mathcal H,i}^{(j)}$ with $\alpha_{\mathcal E,i}^{(k)}<\alpha_{\mathcal H,i}^{(j)}$.
\end{definition}
In particular, a (strongly) parabolic endomorphism of $\mathcal E$ is an endomorphism of $E$ whose restriction to each $E_{p_i}$ is (strictly) upper triangular with respect to the filtration $F_i^\bullet$.  Notation such as $\SL(\mathcal E)$, $\SU(\mathcal E)$, etc.\ should be interpreted as consisting only of endomorphisms which are parabolic at the points in $D$.

\begin{definition}
	An \emph{$\SL(r_\star,\C)$-Higgs field} on a parabolic holomorphic bundle $(\mathcal E,\delbar_E)$ is a global section of $\Omega_C^{(1,0)}\otimes\End_{\mathrm{par},0}(\mathcal E)$, where the subscript $0$ indicates that the endomorphisms should be traceless.
\end{definition}

\begin{definition}
	Let $M_{(Q,\vec r)}$ be the space of all pairs $(\delbar_E,\varphi)$, and consider the map\footnote{This arises as the moment map for the complex gauge group introduced below.}
		\[\mu_{\C,\mathrm{Higgs}}(\delbar_E,\varphi)=\delbar_E\varphi.\]
	For $(\delbar_E,\varphi)\in\mu_{\C,\mathrm{Higgs}}^{-1}(0)$, the residue $\varphi_i=\res_{p_i}\varphi$ induces an endomorphism of the associated graded vector space $\mathrm{gr}F_i^\bullet$, which we shall denote $\mathrm{gr}(\mathrm{res}_{p_i}\varphi)$.
	The parabolic subgroup of $P_i\subset\SL(r_\star,\C)$ decompose into a Levi part $L_i$ and a unipotent part $N_i$, and the Lie algebras decompose $\mathfrak p_i=\mathfrak l_i\oplus\mathfrak n_i$ into Levi and nilpotent subalgebras.  In this terminology, $\res_{p_i} \varphi$ lies in $\mathfrak p_i$, while $\mathrm{gr}$ is the projection to the Levi $\mathfrak l_i$.  We call $\mathrm{gr}(\res_{p_i}\varphi)$ the \emph{Levi part} of $\varphi_i$.  We shall be interested in Higgs fields for which $\mathrm{gr}(\res_{p_i}\varphi)\in Z(\mathfrak l_i)$.  A \emph{central-Levi parabolic Higgs bundle} is one such that $\mathrm{gr}(\res_{p_i}\varphi)\in Z(\mathfrak l_i)$ at all points $p_i \in D$.
\end{definition}

\begin{example}
	Suppose the $i$th branch of $(Q,\vec r)$ is $r_\star=4\leftarrow2\leftarrow1$, and assume the flag is $\C^4\supset\langle e_1,e_2\rangle\supset\langle e_1\rangle$.  Writing
		\[\res_{p_i}\varphi=\begin{pmatrix}a_{11}&*&*&*\\0&b_{11}&*&*\\0&0&c_{11}&c_{12}\\0&0&c_{21}&c_{22}\end{pmatrix}.\]
	the induced map on the associated graded vector space is given by the triple
		\[\mathrm{Lev}(\varphi)_i=\left(\begin{pmatrix}a_{11}\end{pmatrix},\begin{pmatrix}b_{11}\end{pmatrix},\begin{pmatrix}c_{11}&c_{12}\\c_{21}&c_{22}\end{pmatrix}\right)\in\End(\mathrm{gr}F_i^\bullet)\]
	Then $\mathrm{gr}(\res_{p_i}\varphi)\in Z(\mathfrak l_i)$ if and only if $(c_{ij})_{i,j}$ is a scalar matrix.
\end{example}

\begin{remark}
Note that all strongly parabolic Higgs bundles (i.e. $\mathrm{gr}(\res_{p_i}\varphi)=0$) are central-Levi parabolic Higgs bundles. Similarly, note that if $F_i^\bullet$ is a ``full flag'' (i.e.\ each $m_i^{(k)}=1$), then $\mathfrak l_i = Z(\mathfrak l_i)$; hence, in the full flag case, all parabolic Higgs bundles are central-Levi parabolic bundles.
\end{remark}

Because of the natural identification $Z\left(\bigoplus_i\End(\mathfrak l_i)\right)\cong\bigoplus_i\C^{\ell_i}$, this space depends only the dimensions of $F_i^\bullet$ encoded in the dimension vector $\vec r$ of $Q$.''  This will be useful when we allow $\mathcal F$ to vary.

The parabolic structure on $E$ induces one on $\det E$ with weights $\alpha_i^{\det}=\sum_k m_i^{(k)}\alpha_i^{(k)}$.
\begin{definition}\label{def: SL Higgs field}
An \emph{$\SL(r, \C)$-parabolic Higgs bundle} is a rank $r$ parabolic vector bundle $\mathcal E=(E,\vec\alpha,\mathcal F)$, a holomorphic structure $\delbar_E$, and an $\SL(r,\C)$-Higgs field $\varphi$ together with an isomorphism between $\det \mathcal E$ and the trivial bundle with a fixed parabolic structure.
\end{definition}

Note that the reduction of structure group to $\SL(r_\star,\C)$ requires the induced parabolic weights on $\det E$ be integral.  We make the convention to fix the parabolic weights $\alpha_i^{\det}=0$, which imposes the constraint $\sum_k m_i^{(k)}\alpha_i^{(k)}=0$.  A parabolic subbundle $\mathcal S$ of $\mathcal E$ is a parabolic bundle together with a parabolic inclusion $\mathcal S\into\mathcal E$.

\begin{definition}
    A parabolic Higgs bundle $(\mathcal E,\delbar_E,\varphi)$ is $\alpha$-\emph{(semi)stable} if for every proper $\varphi$-invariant parabolic subbundle $\mathcal S\subset\mathcal E$ we have $\textrm{slope}\,\mathcal S<\textrm{slope}\,\mathcal E$ (respectively $\text{slope}\,\mathcal S\leq\text{slope}\,\mathcal E$).
\end{definition}

Every holomorphic subbundle of $\mathcal E$ naturally inherits a parabolic structure with maximal parabolic weights (and thus maximal slope) for which the inclusion map parabolic.\footnote{
	Let $S\subseteq E$ be a holomorphic subbundle.  The induced parabolic structure in the fiber over $p_i$ is
$$\begin{array}{ccccccccc}
		S_{p_i}	&=	&S_{p_i}\cap F_i^{(\ell_i+1)}	&\supset	&\cdots	&\supset	&S_{p_i}\cap F_i^{(1)}		&\supset	&0\\
		&		&\alpha_i^{(\ell_i+1)}	&<			&\cdots	&<			&\alpha_i^{(1)}	&&
\end{array}$$
Whenever there is repetition in the top row, we remove duplicates and inherit the largest corresponding parabolic weight.
}  Hence it suffices to only consider holomorphic subbundles when checking stability of a given parabolic Higgs bundle.

\begin{definition}[Moduli Space of Central-Levi Parabolic Higgs Bundles]
	Fix the data $Q,\vec r,D,E,\mathcal F,\alpha$ as above, and an element $\sigma\in Z\left(\bigoplus_i\End(\mathfrak l_i)\right)\cong\bigoplus_i\C^{\ell_i}$.  Define the map
		\[\mathrm{Lev}\from\mu_{\C,\mathrm{Higgs}}^{-1}(0)\to\bigoplus_i\End(\mathrm{gr}F_i^\bullet),\qquad\mathrm{Lev}(\delbar_E,\varphi)=(\mathrm{gr}(\res_{p_i}(\varphi)))_{i=1}^n\]
	and set $M_{(Q,\vec r,\mathcal F)}(\sigma)=\mathrm{Lev}^{-1}(\sigma)$.  Let $M_{(Q,\vec r,\mathcal F)}(\sigma,\alpha)$ be the $\alpha$-stable locus of $M_{(Q,\vec r,\mathcal F)}(\sigma)$.  The \emph{moduli space of parabolic $\SL(r_\star,\C)$-Higgs bundles of type $(Q,\vec r)$} is
		\[\mathcal M_{(Q,\vec r)}(\sigma,\alpha)=M_{(Q,\vec r)}(\sigma,\alpha)/\mathfrak G_\C\]
	where $\mathfrak G_\C=C^\infty(C,\SL(\mathcal E))$ is the complex gauge group of parabolic endomorphisms of $\mathcal E$ with determinant 1.  We give $\mathcal M_{(Q,\vec r)}(\sigma,\alpha)$ the quotient topology.
\end{definition}

\begin{remark}\label{rem: changing flags}
	Given a different arrangement of flags $\mathcal F'$ of type $(Q,\vec r)$, one can pick a global section $g_{\mathcal F\leftarrow\mathcal F'}\in C^\infty(C,\SL(E))$ which takes $\mathcal F'$ to $\mathcal F$, and use it to create a map of configuration spaces $M_{(Q,\vec r,\mathcal F')}(\sigma,\alpha)\to M_{(Q,\vec r,\mathcal F)}(\sigma,\alpha)$.  Although there are many choices for this $g_{\mathcal F\leftarrow\mathcal F'}$, the induced map between moduli spaces is uniquely determined.  Therefore we can refer to the moduli space $\mathcal M_{(Q,\vec r)}(\sigma,\alpha)$ without ambiguity.  This is an important detail, as the construction of the map from the quiver variety $\mathcal X$ in the next section allows the flags to vary.
\end{remark}

\begin{theorem}[Konno \cite{Kon93}, Biquard-Boalch \cite{BiquardBoalch}, Collier-Fredrickson-Wentworth \cite{CFW24}]
	Given generic parabolic weights, the moduli space constructed above has the structure of a smooth hyperk\"ahler manifold.\footnote{
		Konno addresses the strongly parabolic case.  The weakly parabolic full flag case is a special case of Biquard and Boalch's results for wild Higgs bundles.  Collier-Fredrickson-Wentworth provides a different proof in the weakly parabolic full flag case, and their proof generalizes without modification to our central-Levi parabolic Higgs bundle moduli spaces.
	}
\end{theorem}
The hyperk\"ahler structure provides a distinguished holomorphic symplectic form $\Omega_\mathrm{Higgs}$ which we describe in the next subsection.

Since the underlying Riemann surface is $\CP^1$, the Birkhoff--Grothendieck decomposition gives a stratification of the parabolic Hitchin moduli space $\mathcal M=\mathcal M_{(Q,\vec r)}(\sigma,\alpha)$ according to the underlying holomorphic structures of its elements.  

\begin{definition}
	Let $\mathcal M_0\subseteq\mathcal M$ be the subset of consisting of Higgs bundles whose holomorphic structure is isomorphic to $\mathcal O^{\oplus r_\star}$.
\end{definition}

This is a Zariski open subset, as the subspace of the configuration space $M_{(Q,\vec r)}$ consisting of Higgs bundles with trivial holomorphic structures is open.\footnote{
	This classical result follows from the semicontinuity of the function $\delbar_E\mapsto\dim H^0(C,(E,\delbar_E))$.
}

\subsection{Holomorphic Symplectic Structure}\label{subsec: Omega construction on Hitchin moduli space}
Next, we describe the holomorphic symplectic form $\Omega_\mathrm{Higgs}$ on $\mathcal M=\mathcal M_{(Q,\vec r)}(\sigma,\alpha)$.

Let $(\,\delbar_A,\Phi)\in\mathcal M$ and let $(\,\delbar_{A_j(t)},\Phi_j(t))$ be 1-parameter families of Higgs bundles with $(\,\delbar_{A_j(0)},\Phi_j(0))=(\,\delbar_A,\Phi)$ for $j=1,2$, which we identify with their lifts in the configuration space $M_{(Q,\vec r,\mathcal F)}(\sigma,\alpha)$.  Let $(\dot A_j,\dot\Phi_j)$ be the first variation.  Then the distinguished holomorphic symplectic form $\Omega_\mathrm{Higgs}$ is given by (see \cite{CFW24})
\begin{equation}\label{eqn: Omega Higgs definition}
	\Omega_\mathrm{Higgs}((\dot A_1,\dot\Phi_1),(\dot A_2,\dot\Phi_2))=-\mathbbm i\int_C\tr\left(\dot A_1\wedge\dot\Phi_2-\dot A_2\wedge\dot\Phi_1\right).
\end{equation}

In the strongly parabolic case $\sigma=0$, this $\Omega_\mathrm{Higgs}$ is the exterior derivative of the tautological 1-form given by
\begin{equation}\label{eqn: L Higgs definition}
	L_\mathrm{Higgs}(\dot A,\dot\Phi)=-\mathbbm i\int_C\tr(\dot A\wedge\Phi)
\end{equation}
where $\Phi=\Phi(0)$.

\begin{remark}\label{rem: harmonic representative not necessary}
	In \cite{CFW24}, $(\dot A_j,\dot\Phi_j)$ are \emph{harmonic representatives} of the deformation, meaning they are orthogonal to the $\mathfrak G_\C$-orbits in addition to infinitesimally preserving $\mu_\C=0$ and $\mathrm{Lev}=\sigma$.  However, the gauge orthogonality is not required.  We prove this invariance in \Cref{subsec: Omega independence}.
\end{remark}
These formulas require a given family $(\,\delbar_{A(t)},\Phi(t))\in M_{(Q,\vec r,\mathcal F)}(\sigma,\alpha)$ to have fixed parabolic flags.  In particular, the Higgs field deformations $\dot\Phi$ must be strongly parabolic (Levi part 0).
In \Cref{sec: holomorphic symplectic forms}, we use families of quiver representations to produce families of Higgs bundles $(\,\delbar_E,\mathcal F(t),\varphi(t))$ in which $\delbar_E$ is fixed but the parabolic flag data varies.  In order to then evaluate $\Omega_\mathrm{Higgs}$, we will have to consider families of gauge changes $g_{\mathcal F\leftarrow\mathcal F(t)}$ which takes $(\,\delbar_E,\mathcal F(t),\varphi(t))$ to $(\,\delbar_{A(t)},\mathcal F,\Phi(t))$ with $\mathcal F=\mathcal F(0)$ fixed as in \Cref{rem: changing flags}.  Since $\Omega_\mathrm{Higgs}$ only depends on the first variation of $(\,\delbar_{A(t)},\Phi(t))$, it will suffice to take $g_{\mathcal F\leftarrow\mathcal F(t)}=\exp(t\dot\nu)=\Id+t\dot\nu+O(t^2)$ for an appropriate $\dot\nu$ which we now characterize.

\begin{definition}[First Variation of Flag Data]\label{def: flag changing endomorphism}
	Let $\mathcal F(t)$ be a smooth family of flag data on $(E,\delbar_E)$, and let $\mathcal E(t)=(E,\mathcal F(t),\alpha)$.  We say $\dot\nu\in C^\infty(C,\End(E))$ \emph{induces the first variation} of $\mathcal F(t)$ if for any family $(\,\delbar_E,\varphi(t))\in M_{(Q,\vec r,\mathcal F(t))}(\sigma)$, we have
		\[(e^{-t\dot\nu}\circ\delbar_E\circ e^{t\dot\nu},e^{-t\dot\nu}\varphi(t)e^{t\dot\nu})\in M_{(Q,\vec r,\mathcal F(0))}(\sigma)\]
	up to an $O(t^2)$ correction.
\end{definition}

\begin{lemma}\label{lemma: infinitesimal flag change}
	An endomorphism $\dot\nu$ induces the first variation of the flag data $\mathcal F(t)$ if and only if for any family of holomorphic $\varphi(t)\in\End_{\mathrm{par},0}(\mathcal E(t))$ with $\mathrm{Lev}(\delbar_E,\varphi)=\sigma$ we have $\dot\varphi+[\varphi,\dot\nu]\in\End_{\mathrm{spar},0}(\mathcal E)$. 
\end{lemma}
\begin{proof}
	Suppose $(\,\delbar_E,\varphi(t))$ is a smooth family in $M_{(Q,\vec r,\mathcal F(t))}(\sigma)$ with $\varphi(t)=\varphi+t\dot\varphi+O(t^2)$.  We have
		\[e^{-t\dot\nu}\varphi(t)e^{t\dot\nu}=\varphi+t\left(\dot\varphi+[\varphi,\dot\nu]\right)+O(t^2).\]
	Since $\varphi$ satisfies the residue conditions $\mathrm{gr}(\res_{p_i}\varphi)=\sigma_i$, the same is true of $e^{-t\dot\nu}\varphi(t)e^{t\dot\nu}$ (up to first order in $t$) if and only if $\dot\varphi+[\varphi,\dot\nu]$ is strongly parabolic.
\end{proof}

In the proof of our main theorem, we will only be concerned about describing $\Omega_\mathrm{Higgs}$ on the locus $\mathcal M_0$ of Higgs bundles with holomorphically trivial underlying structure.  On this open subspace, we will consider 1-parameter families of Higgs bundles in which $\delbar_E$ is held fixed, and translate them to families where $\mathcal F$ is held fixed instead.  This motivates the following definition.

\begin{definition}
	Let $(\,\delbar_E,\mathcal F(t),\varphi(t))$ represent a smooth family in $\mathcal M_{(Q,\vec r)}(\sigma,\alpha)$.  Say $\varphi(t)=\varphi+t\dot\varphi+O(t^2)$, and suppose $\dot\nu$ induces the first variation of $\mathcal F(t)$.  We call $(\dot\nu,\dot\varphi)$ a \emph{holomorphic deformation} of $(\,\delbar_E,\mathcal F,\varphi)=(\,\delbar_E,\mathcal F(0),\varphi(0))$.  Conjugation by $e^{-t\dot\nu}$ takes $(\,\delbar_E,\mathcal F(t),\varphi(t))$ to some $(\,\delbar_{A(t)},\Phi(t))\in M_{(Q,\vec r,\mathcal F(0))}(\sigma,\alpha)$, whose first variation $(\dot A,\dot\Phi)$ we call a \emph{fixed-flag deformation} of $(\,\delbar_E,\mathcal F,\varphi)$.  These two deformations are \emph{equivalent} in the sense that they represent the same family in $\mathcal M_{(Q,\vec r)}(\sigma,\alpha)$.\footnote{
		In previous literature (e.g.\ \cite{FMSW21,FMSW26,FY26}), the notation $(\dot\eta,\dot\varphi,\dot\nu)$ represents a deformation of a Hermitian-Yang-Mills triple $(\delbar_E,\varphi,h)$.  This $\dot\nu$ can be thought of as a correction term used to find the harmonic representative $(\dot A,\dot\Phi)$ of the deformation discussed in \Cref{rem: harmonic representative not necessary}.  Our flag correction was inspired by this, but distinct since we never fix a hermitian metric and do not require $\dot\nu$ to solve a PDE.
	}
\end{definition}

Note that the choice of $\dot\nu$ is not at all unique.  When we write $(\dot\nu,\dot\varphi)$, we simply use $\dot\nu$ to keep track of the first variation of $\mathcal F$ which is encoded in $\dot\nu(p_i)$ for $i=1,\dots,n$.  The alternative is to consider $\dot{\mathcal F}$ as a tangent vector in a product of flag varieties and to write $(\dot{\mathcal F},\dot\varphi)$ for holomorphic deformations.

\begin{proposition}\label{prop: Omega on holo deformation}
	Let $(\dot\nu_1,\dot\varphi_1),(\dot\nu_2,\dot\varphi_2)$ be holomorphic deformations of $(\,\delbar_E,\mathcal F,\varphi)$, and let $(\dot A_1,\dot\Phi_1),(\dot A_2,\dot\Phi_2)$ be the equivalent fixed-flag deformations.  Then
		\[\Omega_\mathrm{Higgs}((\dot\nu_1,\dot\varphi_1),(\dot\nu_2,\dot\varphi_2))=-\mathbbm i\int_C\tr\left(\,\delbar_E\dot\nu_1\wedge(\dot\varphi_2+[\varphi,\dot\nu_2])-\delbar_E\dot\nu_2\wedge(\dot\varphi_1+[\varphi,\dot\nu_1])\right).\]
	For a single deformation $(\dot\nu,\dot\varphi)$, the tautological 1-form is
		\[L_\mathrm{Higgs}(\dot\nu,\dot\varphi)=-\mathbbm i\int_C\tr(\,\delbar_E\dot\nu\wedge\varphi).\]
\end{proposition}
\begin{proof}
	This follows from the formulas \eqref{eqn: Omega Higgs definition} and \eqref{eqn: L Higgs definition} for $\Omega_\mathrm{Higgs}$ and $L_\mathrm{Higgs}$, and the fact that
		\[\dot A=\frac{\de}{\de t}\Big|_{t=0}e^{-t\dot\nu}\circ\delbar_E\circ e^{t\dot\nu}=\delbar_E\dot\nu,\]
		\[\dot\Phi=\frac{\de}{\de t}\Big|_{t=0}e^{-t\dot\nu}\varphi(t)e^{t\dot\nu}=\dot\varphi+[\varphi,\dot\nu]\]
	for a given holomorphic deformation $(\dot\nu,\dot\varphi)$.
\end{proof}

\section{Map from Nakajima Quiver Varieties to Higgs Bundle Moduli Spaces}\label{sec: the map}
In this section, we construct a map $\mathcal T$ from a suitable star-shaped Nakajima quiver varieties to a corresponding moduli space of parabolic Higgs bundles.  We begin  with a construction of Higgs bundles from quiver representation data.  We use the complex GIT parameters $\tau_i^{(k)}$ to determine the components $\sigma_i^{(k)}$ of $\sigma\in Z\left(\bigoplus_i\mathfrak l_i\right)$ and prove compatibility for our Higgs field.  We similarly use the real GIT parameters $\beta_i^{(k)}$ for the quiver variety to determine parabolic weights $\alpha_i^{(k)}$.  We prove $\mathcal T$ is well-defined by showing that $\beta$-stable quiver representations map to $\alpha$-stable Higgs bundles, and then show that $\mathcal T$ is a homeomorphism onto the subspace $\mathcal M_0\subseteq\mathcal M$ of Higgs bundles with underlying holomorphic structure $(E,\delbar_E)=\mathcal O^{\oplus r_\star}$.  This section takes inspiration from \cite{RS21}, which outlines a similar construction for the case $\tau,\sigma=0$ with the additional restriction that $\beta_i^{(k)}=0$ for $k=2,\dots,\ell_i$.  We make some adaptations to account for some delicate stability issues.\footnote{They do not prove the map sends stable quiver representations to stable Higgs bundles.  Indeed, additional hypotheses on $\beta$, $\alpha$ are needed to guarantee this.}  We also do not consider the more general ``comet-shaped quivers'' which Rayan and Schaposnik discuss, as we have not been able to verify that the map to parabolic Hitchin moduli spaces on higher genus surfaces is well-defined.

\subsection{\texorpdfstring{Construction of $\mathcal T$}{Construction of T}}

Recall that each $\tau_i^{(k)}$ belongs to the center $Z(\mathfrak{gl}(r_i^{(k)},\C))\cong\C$, and similarly $\sigma_i$ has components $\sigma_i^{(k)}\in Z(\End(F_i^{(k)}/F_i^{(k-1)}))\cong\C$.  We identify these scalar matrices with complex numbers, so that expressions like $\tau_i^{(k)}=\sigma_i^{(k)}-\sigma_i^{(k+1)}$ and $\sum_{k=k_0}^{\ell_i}\tau_i^{(k)}$ make sense.

\begin{definition}[Construction of $\mathcal T$]\label{def: T construction}
	Fix a star-shaped quiver and branch-increasing dimension vector $(Q,\vec r)$ as in \Cref{fig: quiver}, and central elements $\tau,\beta$ as in \Cref{subsec: quiver variety construction}.  Recall the quotient construction $\mathcal X=\mathcal X_{(Q,\vec r)}(\tau,\beta)=\mu_\C^{-1}(\tau)^{\beta-\mathrm{st}}/G_\C$.  Given a representation $(\bx,\by)\in\mu_\C^{-1}(\tau)$, we construct a parabolic Higgs bundle $\mathcal T(\bx,\by)$ as follows.  As before, we take $r_i^{(\ell_i+1)}=r_\star$ and $r_i^{(0)}=0$ to simplify our formulas.  Similarly, we take $x_i^{(\ell_i+1)}=0$, $x_i^{(0)}=0$, etc.  Define
	\begin{equation}\label{eqn: def of X}
		X_i^{(k)}=x_i^{(\ell_i)}\!\circ\cdots\circ x_i^{(k)}.
	\end{equation}
	We give the trivial holomorphic vector bundle $(E,\delbar)=\mathcal O^r\to\P^1$ a parabolic structure as follows.  At each point $p_i\in D$ we define the filtrations
	\begin{equation}\label{eqn: parabolic flags and weights}
		\begin{array}{ccccccccc}
		E_{p_i}	&=	&F_i^{(\ell_i+1)}	&\supset	&\cdots	&\supset	&F_i^{(1)}		&\supset	&0\\
		-\frac12	&<			&\alpha_i^{(\ell_i+1)}	&<			&\cdots	&<			&\alpha_i^{(1)}	&<	&\frac12
		\end{array}
	\end{equation}
	where $F_i^{(k)}=\im(X_i^{(k)})$ and the weights $\alpha_i^{(k)}$ are uniquely determined\footnote{
		Namely, $\alpha_i^{(\ell_i+1)}=-\frac1{r_\star}\sum_{k=1}^{\ell_i}r_i^{(k)}\beta_i^{(k)}$.  The proof is identical to that of \Cref{prop: map construction properties} (ii) below.}
	by
		\[\alpha_i^{(k)}-\alpha_i^{(k+1)}=\beta_i^{(k)},\qquad\qquad \sum_{k=1}^{\ell_i+1}m_i^{(k)}\alpha_i^{(k)}=0,\]
	where $m_i^{(k)}=r_i^{(k)}-r_i^{(k-1)}$ is the multiplicity of $\alpha_i^{(k)}$.  The second equation is required for the determinant condition from \Cref{def: SL Higgs field}.  Similarly, we use $\tau$ to define a tuple of complex masses $\sigma$ by the equations
	\begin{equation}\label{eqn: tau sigma relation}
		\tau_i^{(k)}=\sigma_i^{(k)}-\sigma_i^{(k+1)},\qquad k=1,\dots,\ell_i
	\end{equation}
	and
	\begin{equation}\label{eqn: sigma sum rule}
		\sum_{k=1}^{\ell_i+1}m_i^{(k)}\sigma_i^{(k)}=0.
	\end{equation}
	Finally, to construct the Higgs field we prescribe residues $\varphi_i=(x_i^{(\ell_i)}y_i^{(\ell_i)})_0$ and set
		$$\varphi=\sum_i\frac{\varphi_i\,\de z}{z-p_i}.$$
\end{definition}

\begin{example}\label{ex: 4 - 2 - 1 example}
	Suppose the $i$th branch of $(Q,\vec r)$ is $r_\star=4\leftarrow2\leftarrow1$.  Let $(\bx,\by)\in\mathcal X_{(Q,\vec r)}(\tau,\beta)$.  Up to $G_\C$-equivalence, we may assume
	\begin{equation}
		x_i^{(1)}=\begin{pmatrix}1\\0\end{pmatrix},
			\qquad x_i^{(2)}=\begin{pmatrix}1&0\\0&1\\0&0\\0&0\end{pmatrix}.
	\end{equation}
	Then $\mu_{\C,i}^{(k)}(\bx,\by)=\tau_i^{(k)}$ for $k=1,2$ requires
	\begin{equation}
		y_i^{(1)}=\begin{pmatrix}\tau_i^{(1)}&a\end{pmatrix},
			\qquad y_i^{(2)}=\begin{pmatrix}\tau_i^{(1)}+\tau_i^{(2)}&a&*&*\\0&\tau_i^{(2)}&*&*\end{pmatrix}
	\end{equation}
	where $a\in\C$ and the asterisks are free.  In $\mathcal T(\bx,\by)$, the parabolic flag will be $E_{p_i}\supset\langle e_1,e_2\rangle\supset\langle e_1\rangle$, and the Higgs field residue is
	\begin{equation}
		\varphi_i=\left(x_i^{(2)}y_i^{(2)}\right)_0
			=\begin{pmatrix}\tau_i^{(1)}+\tau_i^{(2)}&a&*&*\\0&\tau_i^{(2)}&*&*\\0&0&0&0\\0&0&0&0\end{pmatrix}
			-\frac14\left(\tau_i^{(1)}+2\tau_i^{(2)}\right)\Id.
	\end{equation}
	The induced map on the associated graded space $\mathrm{gr}F_i^\bullet=\C\oplus\C\oplus\C^2$ is
		\[\mathrm{gr}(\varphi_i)=(\sigma_i^{(1)},\sigma_i^{(2)},\sigma_i^{(3)})=\frac14\left(3\tau_i^{(1)}+2\tau_i^{(2)},-\tau_i^{(1)}+2\tau_i^{(2)},-\tau_i^{(1)}-2\tau_i^{(2)}\right).\]
	Now we see $\sigma_i^{(k)}-\sigma_i^{(k+1)}=\tau_i^{(k)}$ for $k=1,2$, and $\sum_k m_i^{(k)}\sigma_i^{(k)}=\sigma_i^{(1)}+\sigma_i^{(2)}+2\sigma_i^{(3)}=0$.
\end{example}

Before proceeding, prove a useful formula which we encourage the reader to read carefully:
\begin{lemma}[Commutation Rule]\label{lemma: xy commutation rule}
	Let $(\bx,\by)\in\mu_\C^{-1}(\tau)$.  Then for any $k_0=1,\dots,\ell_i$,
		\[y_i^{(\ell_i)}X_i^{(k_0)}=y_i^{(\ell_i)}x_i^{(\ell_i)}\!\cdots x_i^{(k_0)} = x_i^{(\ell_i-1)}\!\cdots x_i^{(k_0-1)}y_i^{(k_0-1)} + \left(\sum_{k=k_0}^{\ell_i}\tau_i^{(k)}\right)x_i^{(\ell_i-1)}\!\cdots x_i^{(k_0)}.\]
	Diagramatically, we represent this relation by
        
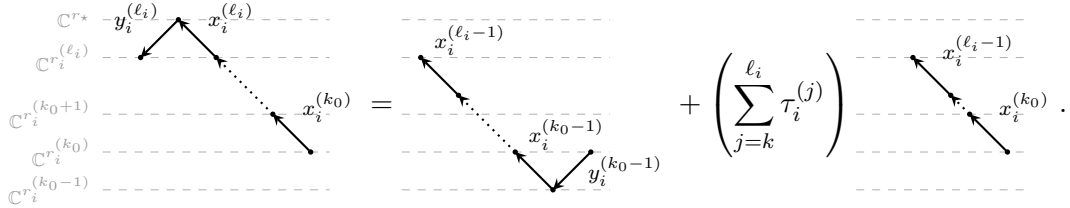
\[
        \begin{tikzpicture}[baseline=15pt, >=stealth, thick, scale=0.5]
            \draw[gray!60, thin, dashed] (-5.5, -1) -- (0.5, -1);
            \node[left, gray!80, font=\tiny] at (-5.5, -1) {$\C^{r_i^{(k_0-1)}}$};
            
            \draw[gray!60, thin, dashed] (-5.5, 0) -- (0.5, 0);
            \node[left, gray!80, font=\tiny] at (-5.5, 0) {$\C^{r_i^{(k_0)}}$};
            
            \draw[gray!60, thin, dashed] (-5.5, 1) -- (0.5, 1);
            \node[left, gray!80, font=\tiny] at (-5.5, 1) {$\C^{r_i^{(k_0+1)}}$};
            
            \draw[gray!60, thin, dashed] (-5.5, 2.5) -- (0.5, 2.5);
            \node[left, gray!80, font=\tiny] at (-5.5, 2.5) {$\C^{r_i^{(\ell_i)}}$};
            
            \draw[gray!60, thin, dashed] (-5.5, 3.5) -- (0.5, 3.5);
            \node[left, gray!80, font=\tiny] at (-5.5, 3.5) {$\C^{r_\star}$};

            \coordinate (n0) at (-1,-1);
            \coordinate (n1) at (0,0);
            \coordinate (n2) at (-1,1);
            \coordinate (n3) at (-2.5,2.5);
            \coordinate (n4) at (-3.5,3.5);
            \coordinate (n5) at (-4.5,2.5);
            
            \fill (n0) circle (0pt);
            \fill (n1) circle (2pt);
            \fill (n2) circle (2pt);
            \fill (n3) circle (2pt);
            \fill (n4) circle (2pt);
            \fill (n5) circle (2pt);

            \draw[->] (n1) -- (n2) node[midway, above right, font=\scriptsize] {$x_i^{(k_0)}$};
            \draw[->, dotted] (n2) -- (n3);
            \draw[->] (n3) -- (n4) node[midway, above right, xshift = 0pt, yshift = -2pt, font=\scriptsize] {$x_i^{(\ell_i)}$};
            \draw[->] (n4) -- (n5) node[midway, above left, xshift = 4pt, yshift = -2pt, font=\scriptsize] {$y_i^{(\ell_i)}$};
        \end{tikzpicture}
        =
        \begin{tikzpicture}[baseline=15pt, >=stealth, thick, scale=0.5]
            \draw[gray!60, thin, dashed] (-5, -1) -- (0.5, -1);
            \draw[gray!60, thin, dashed] (-5, 0) -- (0.5, 0);
            \draw[gray!60, thin, dashed] (-5, 1) -- (0.5, 1);
            \draw[gray!60, thin, dashed] (-5, 2.5) -- (0.5, 2.5);
            \draw[gray!60, thin, dashed] (-5, 3.5) -- (0.5, 3.5);

            \coordinate (n0) at (0,0);
            \coordinate (n1) at (-1,-1);
            \coordinate (n2) at (-2,0);
            \coordinate (n3) at (-3.5,1.5);
            \coordinate (n4) at (-4.5,2.5);
            \coordinate (n5) at (-3.5,3.5);

            \fill (n0) circle (2pt);
            \fill (n1) circle (2pt);
            \fill (n2) circle (2pt);
            \fill (n3) circle (2pt);
            \fill (n4) circle (2pt);
            \fill (n5) circle (0pt);

            \draw[->] (n0) -- (n1) node[midway, right, xshift = 2pt, yshift = -0pt, font=\scriptsize] {$y_i^{(k_0-1)}$};
            \draw[->] (n1) -- (n2) node[midway, above right, xshift = -6pt, yshift = 4pt, font=\scriptsize] {$x_i^{(k_0-1)}$};
            \draw[->, dotted] (n2) -- (n3);
            \draw[->] (n3) -- (n4) node[midway, above right, xshift = -6pt, yshift = 4pt, font=\scriptsize] {$x_i^{(\ell_i-1)}$};
        \end{tikzpicture}
        + \left(\sum_{j=k}^{\ell_i}\tau_i^{(j)}\right)
        \begin{tikzpicture}[baseline=15pt, >=stealth, thick, scale=0.5]
            \draw[gray!60, thin, dashed] (-4, -1) -- (0.5, -1);
            \draw[gray!60, thin, dashed] (-4, 0) -- (0.5, 0);
            \draw[gray!60, thin, dashed] (-4, 1) -- (0.5, 1);
            \draw[gray!60, thin, dashed] (-4, 2.5) -- (0.5, 2.5);
            \draw[gray!60, thin, dashed] (-4, 3.5) -- (0.5, 3.5);

            \coordinate (nInvisible) at (-1,-1);
            \coordinate (n0) at (0,0);
            \coordinate (n1) at (-1,1);
            \coordinate (n2) at (-1.5,1.5);
            \coordinate (n3) at (-2.5,2.5);
            
            \fill (nInvisible) circle (0pt);
            \fill (n0) circle (2pt);
            \fill (n1) circle (2pt);
            \fill (n2) circle (2pt);
            \fill (n3) circle (2pt);

            \draw[->] (n0) -- (n1) node[midway, above right, font=\scriptsize] {$x_i^{(k_0)}$};
            \draw[->, dotted] (n1) -- (n2);
            \draw[->] (n2) -- (n3) node[midway, above right, font=\scriptsize] {$x_i^{(\ell_i-1)}$};
        \end{tikzpicture}.
        \]
	where composition reads right to left, and the height of a node represents the level the flag.
\end{lemma}
\begin{proof}
	By setting the moment map \eqref{eqn: mu C i} equal to $\tau_i^{(k)}$, we get
	\begin{equation}\label{eqn: residue swap rule}
		y_i^{(k)}x_i^{(k)}=x_i^{(k-1)}y_i^{(k-1)}+\tau_i^{(k)}\text{Id}
	\end{equation}
	for all $k=1,\dots,\ell_i$ (as always, we take $x_i^{(0)}=0$ and $y_i^{(0)}=0$).  Applying this iteratively yields the result.
\end{proof}

\begin{proposition}\label{prop: map construction properties}
	The above construction satisfies the following statements:
	\begin{enumerate}[label=(\roman*)]
		\item The Higgs field $\varphi$ is holomorphic at infinity.
		\item At the highest level the complex mass is
		\begin{equation}\label{eqn: highest level complex mass}
			\sigma_i^{(\ell_i+1)}=-\frac1{r_\star}\sum_{k=1}^{\ell_i}r_i^{(k)}\tau_i^{(k)}.
		\end{equation}
		\item The Higgs field is adapted to $(\mathcal F,\sigma)$, i.e.\ $\varphi\in\End_\mathrm{par}(\mathcal E)$ and $\mathrm{gr}(\mathrm{res}_{p_i}(\varphi))=\sigma_i$.
	\end{enumerate}
\end{proposition}
\begin{proof}
	For (i), since $\mu_\C(\bx,\by)=\tau$ the moment map formula \eqref{eqn: mu SL} implies
		\[0=-\mu_{\SL(r,\C)}(\bx,\by)=\sum_{i=1}^n\varphi_i.\]
	A route calculation shows that this sum is the residue of the Higgs field at infinity. Since this residue vanishes, $\varphi$ is holomorphic at infinity.
		
	For (ii), we use the relations between the $\tau$ and $\sigma$ given in \eqref{eqn: tau sigma relation} and \eqref{eqn: sigma sum rule} and reindex the sum to get
	\begin{equation}
		\sum_{k=1}^{\ell_i+1}m_i^{(k)}\sigma_i^{(k)}=\sum_{k=1}^{\ell_i+1}(r_i^{(k)}-r_i^{(k-1)})\sigma_i^{(k)}=r_\star\sigma_i^{(k+1)}+\sum_{k=1}^{\ell_i}r_i^{(k)}(\sigma_i^{(k)}-\sigma_i^{(k+1)})=0.
	\end{equation}
	Since $\sigma_i^{(k)}-\sigma_i^{(k+1)}=\tau_i^{(k)}$, solving for $\sigma_i^{(k+1)}$ yields \eqref{eqn: highest level complex mass}.
	
	To prove (iii), suppose $v\in F_i^{(k_0)}$.  Then by our choice of $F_i^{(k_0)}$, we can write $v=X_i^{(k_0)}w$ for some $w\in\C^{r_i^{(k_0)}}$\!.  Using $(x_i^{(\ell_i)}y_i^{(\ell_i)})_0=x_i^{(\ell_i)}y_i^{(\ell_i)}-\frac1{r_\star}\tr(x_i^{(\ell_i)}y_i^{(\ell_i)})$ and applying \Cref{lemma: xy commutation rule} to commute $y_i^{(\ell_i)}$ past the $x$'s,
	\begin{align}
		\varphi_i v
			&=x_i^{(\ell_i)}y_i^{(\ell_i)}X_i^{(k_0)}w
				-\frac1{r_\star}\tr(x_i^{(\ell_i)}y_i^{(\ell_i)})v
					\notag\\
			&=x_i^{(\ell_i)}x_i^{(\ell_i-1)}\!\cdots x_i^{(k_0-1)}y_i^{(k_0-1)}w + \left(\sum_{k=k_0}^{\ell_i}\tau_i-\frac1{r_\star}\tr(x_i^{(\ell_i)}y_i^{(\ell_i)})\right)v.
				\label{eqn: iterated residue swap}
	\end{align}
	In the last expression, the first term belongs to $F_i^{(k_0-1)}$.  It remains to show the coefficient in the second term is $\sigma_i^{(k)}$.  Applying \eqref{eqn: residue swap rule} iteratively and using the cyclic property of trace,
	\begin{align}
		\tr(x_i^{(\ell_i)}y_i^{(\ell_i)})
			&=\tr(y_i^{(\ell_i)}x_i^{(\ell_i)})
				\notag\\
			&=\tr(x_i^{(\ell_i-1)}y_i^{(\ell_i-1)})+r_i^{(k)}\tau_i^{(k)}
				\notag\\[-5pt]
			&\ \,\vdots
				\notag\\[-12pt]
			&=\sum_{k=1}^{\ell_i}r_i^{(k)}\tau_i^{(k)}
				\notag\\
			&=-r_\star\sigma_i^{(\ell_i+1)}.
				\label{eqn: trace part of phi residue}
	\end{align}
	Thus
		$$\sum_{k=k_0}^{\ell_i}\tau_i - \frac1{r_\star}\tr(x_i^{(\ell_i)}y_i^{(\ell_i)})=\sum_{k=k_0}^{\ell_i}\tau_i+\sigma_i^{(\ell_i+1)}=\sigma_i^{(k_0)}$$
	as desired.
\end{proof}

\subsection{Stability Results}
The next proposition shows that this construction preserves stability.

\begin{theorem}\label{thm: T preserves stability}
	Assume $\tau,\beta$ and $\sigma,\alpha$ are related as above, and assume further that $\sum_{i,k}r_i^{(k)}\beta_i^{(k)}<1$.  If $(\bx,\by)\in\mu_\C^{-1}(\tau)$ is $\beta$-stable, then the Higgs bundle $\mathcal T(\bx,\by)=(E,\delbar_E,\mathcal F,\varphi)$ constructed above is stable.
\end{theorem}
\begin{proof}
	In light of \Cref{prop: King stability}, we let $\tilde\beta$ be the extended GIT parameter on $\tilde G_\C=\GL(r_\star,\C)\!\times\!\prod_{\substack{i=1,\dots,n\\k=1,\dots,\ell_i}}\!\GL(r_i^{(k)},\C)$ from \Cref{def: extended GIT weight vector} with the weight on the central node
	\begin{equation}\label{eqn: beta star again}
		\beta_\star=-\frac1{r_\star}\sum_{\substack{i=1,\dots,n\\k=1,\dots,\ell_i}}r_i^{(k)}\beta_i^{(k)}.
	\end{equation}
	We will prove the contrapositive statement.  Assume $\mathcal H\subseteq\mathcal E$ is a destabilizing (proper) $\varphi$-invariant holomorphic subbundle of $E$.  We may as well assume $\mathcal H$ has the induced parabolic structure, as that has maximal slope.  Say $s_\star=\rk\mathcal H$, $s_i^{(k)}=\dim\mathcal H_{p_i}\cap F_i^{(k)}$ for $k=1,\dots,\ell_i$, $s_i^{(0)}=0$, and $s_i^{(\ell_i+1)}=s_\star$.  The induced parabolic weight multiplicities are $m_{S,i}^{(k)}=s_i^{(k)}-s_i^{(k-1)}$.
	
	Since $\mathcal H$ is a holomorphic subbundle of a trivial bundle, it cannot have positive degree.  In the case where $\deg\mathcal H\leq-1$ we have
	{\allowdisplaybreaks
	\begin{align}
		\mathrm{pardeg}\,\mathcal H
			&=\deg H+\sum_{i=1}^n\sum_{k=1}^{\ell_i+1}m_{S,i}^{(k)}\alpha_i^{(k)}
				\notag\\
			&\leq\deg H+\sum_{i=1}^n\sum_{k=1}^{\ell_i+1}(s_i^{(k)}-s_i^{(k-1)})\alpha_i^{(k)}
				\notag\\
			&=\deg H+\sum_{i=1}^n\left(s_i^{(\ell_i+1)}\alpha_i^{(\ell_i+1)}+\sum_{k=1}^{\ell_i}s_i^{(k)}(\alpha_i^{(k)}-\alpha_i^{(k+1)})\right)
				\notag\\
			&=\deg H+\sum_{i=1}^n\left(s_i^{(\ell_i+1)}\alpha_i^{(\ell_i+1)}+\sum_{k=1}^{\ell_i}s_i^{(k)}\beta_i^{(k)}\right)
				\notag\\
			&<-1+\sum_{i=1}^n\sum_{k=1}^{\ell_i}r_i^{(k)}\beta_i^{(k)}
				\notag\\
			&<0
	\end{align}
	}by our hypothesis $\sum_{i,k}r_i^{(k)}\beta_i^{(k)}<1$, and since each $s_i^{(k)}\leq r_i^{(k)}$ and $\alpha_i^{(\ell_i+1)}<0$.  Since $\mathrm{slope}\,\mathcal E=0$, this implies $\mathcal H$ does not destabilize $\mathcal E$, a contradiction.  Let us therefore assume $\deg\mathcal H=0$.  By the Birkhoff-Grothendiek theorem, $\mathcal H$ must be trivial.\footnote{
		Indeed, if $\mathcal H=\mathcal O(d_1)\oplus\cdots\mathcal O(d_{s_\star})$ is a subbundle of $(E,\delbar_E)\cong\mathcal O^{\oplus r_\star}$, then each $d_i\leq0$.  But $\sum_i d_i=\deg\mathcal H=0$, so $\mathcal H=\mathcal O^{\oplus s_\star}$.
	}  Say $\mathcal H=\mathcal O_C\otimes V$ for some $s_\star$-dimensional vector space $V$.  Because $\mathcal H$ destabilizes $\mathcal E$, just as above we have
	\begin{equation}
		0<\mathrm{pardeg}\,\mathcal H=\sum_{i=1}^n\left(s_i^{(\ell_i+1)}\alpha_i^{(\ell_i+1)}+\sum_{k=1}^{\ell_i}s_i^{(k)}\beta_i^{(k)}\right).
	\end{equation}
	Consider the subrepresentation $S$ of $\tilde Q$ with vector space $S_\star=V$ at the central node and with preimages $S_i^{(k)}=(X_i^{(k)})^{-1}(V)$ at the remaining nodes, with arrows obtained by restricting the $x_i^{(k)}$ and $y_i^{(k)}$.  The dimension vector of $S$ is $\vec s=\left(s_\star,(s_i^{(k)})_{i,k}\right)$. 
	\begin{claim}\label{claim: subrep well defined}
		This subrepresentation $S$ is well-defined.
	\end{claim}
	\begin{proof}[Proof of \Cref{claim: subrep well defined}]
		It is clear that $x_i^{(k)}$ maps $S_i^{(k)}$ into $S_i^{(k+1)}$.  If $v\in S_i^{(k)}$, then $X_i^{(k)}v\in V$.  To prove $y_i^{(k-1)}v\in S_i^{(k-1)}$, we must show $X_i^{(k-1)}y_i^{(k-1)}v\in V$.  Applying the commutation rule \Cref{lemma: xy commutation rule}, then $\tau_i^{(k)}=\sigma_i^{(k)}-\sigma_i^{(k+1)}$ from \eqref{eqn: tau sigma relation}, then the formula \eqref{eqn: trace part of phi residue} for $\tr(x_i^{(\ell_i)}y_i^{(\ell_i)})$ yields
		\begin{align}
			X_i^{(k-1)}y_i^{(k-1)}v
				&=x_i^{(\ell_i)}y_i^{(\ell_i)}X_i^{(k)}v-\left(\sum_{j=k}^{\ell_i}\tau_i^{(j)}\right)X_i^{(k)}v
					\notag\\
				&=x_i^{(\ell_i)}y_i^{(\ell_i)}X_i^{(k)}v+\left(\sigma_i^{(\ell_i+1)}-\sigma_i^{(k)}\right)X_i^{(k)}v
					\notag\\
				&=\left(x_i^{(\ell_i)}y_i^{(\ell_i)}\right)_0X_i^{(k)}v-\sigma_i^{(k)}X_i^{(k)}v,
		\end{align}
		which belongs to $V$ since $\varphi_i=\left(x_i^{(\ell_i)}y_i^{(\ell_i)}\right)_0$ preserves $V$.
	\end{proof}
	Since each $x_i^{(k)}$ is injective (\Cref{lemma: the x are injective}), the dimension vector $\vec s$ of $S$ is given by $s_\star$ and the $s_i^{(k)}$ from above.  Observe that
	\begin{align}
		r_\star\beta_\star
			&=-\sum_{\substack{i=1,\dots,n\\k=1,\dots,\ell_i}}r_i^{(k)}\beta_i^{(k)}
				\notag\\
			&=-\sum_{\substack{i=1,\dots,n\\k=1,\dots,\ell_i}}r_i^{(k)}(\alpha_i^{(k)}-\alpha_i^{(k+1)})
				\notag\\
			&=\sum_{i=1}^n\left(r_i^{(\ell_i)}\alpha_i^{(\ell_i+1)}-\sum_{k=1}^{\ell_i}(r_i^{(k)}-r_i^{(k-1)})\alpha_i^{(k)}\right)
				\notag\\
			&=\sum_{i=1}^n\left(r_i^{(\ell_i)}\alpha_i^{(\ell_i+1)}+m_i^{(\ell_i+1)}\alpha_i^{(\ell_i+1)}-\sum_{k=1}^{\ell_i+1}m_i^{(k)}\alpha_i^{(k)}\right)
				\notag\\
			&=\sum_{i=1}^n r_\star\alpha_i^{(\ell_i+1)}
	\end{align}
	(the last equality uses $\sum_{k=1}^{\ell_i+1}m_i^{(k)}\alpha_i^{(k)}=0$ and $m_i^{(\ell_i+1)}=r_\star-r_i^{(\ell_i)}$).  Lastly, we calculate
	\vspace{5pt}
	\begin{align}
		\tilde\beta(S)
			=s_\star\beta_\star+\sum_{\substack{i=1,\dots,n\\k=1,\dots,\ell_i}}s_i^{(k)}\beta_i^{(k)}
			=\sum_{i=1}^n\left(s_\star\alpha_i^{(\ell_i+1)}+\sum_{k=1}^{\ell_i}s_i^{(k)}\beta_i^{(k)}\right)
			=\mathrm{slope}\,\mathcal H
			>0
	\end{align}
	so $(\bx,\by)$ is unstable by \Cref{prop: King stability}.
\end{proof}

\begin{remark}
	For generic $\beta,\alpha$, the notions of stability and semistability coincide.  In fact, both parameter spaces have chamber structures which in general only match up in the half space defined by $\sum_{i,k}r_i^{(k)}\beta_i^{(k)}<1$.  Hence this hypothesis is necessary for \Cref{thm: T preserves stability} to hold.  Although we chose $\alpha$ to satisfy $\alpha_i^{(k)}-\alpha_i^{(k+1)}=\beta_i$ for the construction of $\mathcal T$, all the results hold when we vary $\alpha$ and $\beta$ within their given chambers.  See \cite[Section 3]{FY26} for a thorough exploration in the $n=4$, $r_\star=2$, strongly parabolic case.
\end{remark}

If $(\bx,\by)$ and $(\bx',\by')$ are $G_\C$-equivalent elements of $\mu_\C^{-1}(\tau)$, our construction produces Higgs bundles $\mathcal T(\bx,\by),\mathcal T(\bx',\by')$ which are equivalent up to a constant $g\in H^0(C,\SL(E))$.  Although the construction of $\mathcal M_{(Q,\vec r)}(\sigma,\alpha)=M_{(Q,\vec r,\mathcal F)}(\sigma,\alpha)/\mathfrak G_\C$ in \Cref{sec: Higgs bundle moduli space construction} holds the flags fixed, in light of \Cref{rem: changing flags} each $(\bx,\by)\in\mathcal X_{(Q,\vec r)}(\tau,\beta)$ unambiguously determines an element of $\mathcal M_{(Q,\vec r)}(\sigma,\alpha)$ regardless of which flags are used in the construction.  We have therefore created a map of moduli spaces $\mathcal T\from\mathcal X_{(Q,\vec r)}(\tau,\beta)\to\mathcal M_{(Q,\vec r)}(\sigma,\alpha)$.

\subsection{\texorpdfstring{Properties of $\mathcal T$}{Properties of T}}
\begin{theorem}\label{thm: T is a homeomorphism onto its image}
	The map $\mathcal T\from\mathcal X_{(Q,\vec r)}(\tau,\beta)\to\mathcal M_{(Q,\vec r)}(\sigma,\alpha)$ is a homeomorphism onto its image, which is the subspace $\mathcal M_0$ consisting of Higgs bundles with underlying holomorphic structure isomorphic to $\mathcal O^{\oplus r_\star}$.
\end{theorem}

\begin{proof}
	We construct an inverse map.  Let $(E,\delbar_E,\mathcal F,\varphi)\in\mathcal M_0\subseteq\mathcal M_{(Q,\vec r)}(\sigma,\alpha)$.  Since $(E,\delbar_E)$ is isomorphic to $\mathcal O^{\oplus r_\star}$, we may as well pick a trivialization $E=\CP^1\times\C^{r_\star}$ in which $\delbar_E$ is the standard one.  Choose a basis to identify $\mathcal F_i^{(k)}\cong\C^{r_i^{(k)}}$ for $i=1,\dots,n$ and $k=1,\dots,\ell_i+1$.  To define $\bx=(x_i^{(k)})\in\Rep Q$, we let $x_i^{(k)}$ be the inclusion $F_i^{(k)}\into F_i^{(k+1)}$.  Let $\varphi_i=\mathrm{res}_{p_i}\varphi$.  Since the image of $\varphi_i-\sigma_i^{(\ell_i+1)}\mathrm{Id}$ is contained in $F_i^{(\ell_i)}$, the map factors through the inclusion $x_i^{(\ell_i)}$, so we have
		\[\varphi_i-\sigma_i^{(\ell_i+1)}\mathrm{Id}=x_i^{(\ell_i)}y_i^{(\ell_i)}\]
	for a unique $y_i^{(\ell_i)}\from\C^{r_i^{(\ell_i)}}\to\C^{r_i^{(\ell_i-1)}}$.  Similarly, we can write
		\[\left(\varphi_i-\sigma_i^{(k+1)}\mathrm{Id}\right)X_i^{(k+1)}=X_i^{(k)}y_i^{(k)}\]
	for some $y_i^{(k)}\from\C^{r_i^{(k+1)}}\to\C^{r_i^{(k)}}$ (recall our notation $X_i^{(k)}=x_i^{(\ell_i)}\cdots x_i^{(k)}$).  Call this $(\bx,\by)=\mathcal T^{-1}(E,\delbar_E,\mathcal F,\varphi)$.  Since $\varphi=\sum\frac{\varphi_i}{z-p_i}\de z$ is holomorphic at infinity,\footnote{
		Higgs fields of this form span an $n-1$ dimensional subspace of $H^0(C,K_C\otimes O(D))$.  Meanwhile $H^0(C,K_C\otimes\mathcal O(D))\cong H^0(C,\mathcal O(n-2))\cong\C^{n-1}$, so every global section is of this form.
	}
	\begin{equation}
		\mu_{\C,\star}(\bx,\by)=-\sum_i(x_i^{(\ell_i)}y_i^{(\ell_i)})_0=-\sum_i\left(\varphi_i-\sigma_i^{(k+1)}\mathrm{Id}\right)_0=-\sum_i\varphi_i=0.
	\end{equation}
	To evaluate the remaining components of $\mu_\C$, we compute
	\begin{align}
		X_i^{(k)}\mu_{\C,j}(\bx,\by)
			&=X_i^{(k)}\left(y_j^{(k)}x_j^{(k)}-x_i^{(k-1)}y_j^{(k-1)}\right)
				\notag\\
			&=\left(\varphi_i-\sigma_i^{(k+1)}\mathrm{Id}\right)X_i^{(k)}-\left(\varphi_i-\sigma_i^{(k)}\mathrm{Id}\right)X_i^{(k)}
				\notag\\
			&=X_i^{(k)}\left(\sigma_i^{(k)}-\sigma_i^{(k+1)}\right)\mathrm{Id}.
	\end{align}
	Since $X_i^{(k)}$ is injective and $\sigma_i^{(k)}-\sigma_i^{(k+1)}=\tau_i^{(k)}$ (as complex numbers), we have $\mu_{\C,j}(\bx,\by)=\tau_i^{(k)}\mathrm{Id}$.  Therefore $(\bx,\by)\in\mu_\C^{-1}(\tau)$.

	Our choice of trivialization of $(E,\delbar_E)=\mathcal O^{\oplus r_\star}$ determined the gauge up to the action of $H^0(C,\SL(E))\cong\SL(r_\star,\C)$, and our choice of basis for each $F_i^{(k)}$ is absorbed by the remaining factors of the group $G_\C=\SL(r_\star,\C)\times\prod\GL(r_i^{(k)},\C)$.  Thus, the $G_\C$-equivalence class of the $(\bx,\by)$ we constructed is uniquely determined.
	
	To verify $\mathcal T\circ\mathcal T^{-1}=\mathrm{Id}_\mathcal M$, we calculate $\mathcal T(\bx,\by)=(E,\delbar_E',\mathcal F',\varphi')$ for the $(\bx,\by)$ we just constructed.  By construction, $\delbar_E'=\delbar_E$ and $\mathcal F'=\mathcal F$, and
	\begin{equation}
		\res_{p_i}\varphi'
			=(x_i^{(\ell_i)}y_i^{(\ell_i)})_0
			=\left(\varphi_i-\sigma_i^{(k+1)}\mathrm{Id}\right)_0=\varphi_i.
	\end{equation}
	Since $\varphi'$ is determined by its residues, we have $\varphi'=\varphi$.

	Now we verify $\mathcal T^{-1}\circ\mathcal T=\Id_{\mathcal X}$.  Let $(\bx,\by)\in\mathcal X_{(Q,\vec r)}(\tau,\beta)$.  Put $\mathcal T(\bx,\by)=(E,\delbar_E,\mathcal F,\varphi)$, and let $(\bx',\by')$ be constructed from $(E,\delbar_E,\mathcal F,\varphi)$ as above.  In this construction, we chose identifications $F_i^{(k)}\cong\C^{r_i^{(k)}}$, and these choices did not affect the equivalence class of the resulting quiver representation.  But $F_i^{(k)}=\Im(X_i^{(k)})$ and each $X_i^{(k)}$ is injective (\Cref{lemma: the x are injective}), so we may as well pick bases so that $x_i'^{(k)}=x_i^{(k)}$.\footnote{
		First, we choose a basis on $F_i^{(\ell_i)}$ so that the inclusion $F_i^{(\ell_i)}\into E_{p_i}$ is given by $x_i^{(\ell_i)}$.  We recursively define the remaining $x_i^{(k)}$ (with $k$ decreasing) by picking a basis on $F_i^{(k)}$ so that the inclusion into $F_i^{(k+1)}$ (whose basis we have already chosen) is given by $x_i^{(k)}$.
	}
	As for the $y_i'^{(k)}$, we have
	\begin{align}
		X_i^{(k)}y_i'^{(k)}
			&=\left(\varphi_i-\sigma_i^{(k+1)}\Id\right)X_i^{(k+1)}
				\notag\\
			&=\left(x_i^{(\ell_i)}y_i^{(\ell_i)}-\frac1{r_\star}\tr(x_i^{(\ell_i)}y_i^{(\ell_i)})-\sigma_i^{(k+1)}\Id\right)X_i^{(k+1)}.
				\label{eqn: X y prime}
	\end{align}
	By the commutation rule (\Cref{lemma: xy commutation rule}) we have
		\[x_i^{(\ell_i)}y_i^{(\ell_i)}X_i^{(k+1)}=X_i^{(k)}y_i^{(k)}+X_i^{(k+1)}\left(\sum_{j=k+1}^{\ell_i}\tau_i^{(j)}\right).\]
	We calculated in \eqref{eqn: trace part of phi residue} that $\tr(x_i^{(\ell_i)}y_i^{(\ell_i)})=-r_\star\sigma_i^{(\ell_i)}$.  Thus \eqref{eqn: X y prime} becomes
	\begin{align}
		X_i^{(k)}y_i'^{(k)}=X_i^{(k)}\left(y_i^{(k)}+\left(\sum_{j=k+1}^{\ell_i}\tau_i^{(j)}\right)+\sigma_i^{(\ell_i)}-\sigma_i^{(k+1)}\right)=X_i^{(k)}y_i^{(k)}
	\end{align}
	as $\tau_i^{(j)}=\sigma_i^{(j)}-\sigma_i^{(j+1)}$.  Since $X_i^{(k)}$ is injective, this implies $y_i'^{(k)}=y_i^{(k)}$.  This completes the proof.
\end{proof}

\section{Holomorphic Symplectic Forms}\label{sec: holomorphic symplectic forms}
In this section we prove our main theorem:
\begin{theorem}\label{thm: holo symplectic forms agree}
	Let $\mathcal T\from\mathcal X_{(Q,\vec r)}(\tau,\beta)\to\mathcal M_{(Q,\vec r)}(\sigma,\alpha)$ be the map from the Nakajima quiver variety to the parabolic Hitchin moduli space constructed in \Cref{def: T construction} and subject to the hypothesis of \Cref{thm: T preserves stability}.  Let $\Omega_\mathrm{quiv}$ and $\Omega_\mathrm{Higgs}$ be the distinguished holomorphic symplectic forms on $\mathcal X$ and $\mathcal M$, respectively.  Then $\mathcal T^*\Omega_\mathrm{Higgs}=2\pi\Omega_\mathrm{quiv}.$
\end{theorem}

We will prove this theorem by considering the pushforward of deformations $(\dot\bx,\dot\by)$ of a quiver representation $(\bx,\by)\in\mu_\C^{-1}(\tau)\cap\mu_\R^{-1}(\beta)$.  In the \Cref{subsec: deformation of parabolic flag data} we find the endomorphism $\dot\nu$ which induces the first variation of the parabolic flag data in the sense of \Cref{def: flag changing endomorphism}.  In \Cref{subsec: strongly parabolic case via tautological 1-form}, we prove the theorem in the strongly parabolic case $\sigma=0,\tau=0$ by showing the tautological 1-forms satisfy $\mathcal T^*L_\mathrm{Higgs}=2\pi L_\mathrm{quiv}$.  In \Cref{subsec: weakly parabolic case} we perform a much longer calculation to prove the theorem in the general case.  Although the latter implies the former, we include the tautological 1-form calculation as it is far more transparent.

\subsection{Deformation of Parabolic Flag Data}\label{subsec: deformation of parabolic flag data}

Suppose $(\bx(t),\by(t))=(\bx,\by)+t(\dot\bx,\dot\by)+O(t^2)\in\mu_\C^{-1}(\tau)\cap\mu_\R^{-1}(\beta)$, where $(\bx,\by)\in\mu_\C^{-1}(\tau)\cap\mu_\R^{-1}(\beta)$ and $(\dot\bx,\dot\by)$ is a unitary deformation of $(\bx,\by)$ (see \Cref{def: unitary deformation of quiver rep}).  Differentiating the components of the equation $\mu_\C(\bx(t),\by(t))=\tau$ (see \eqref{eqn: mu C i}) gives
\begin{equation}\label{eqn: de mu equation}
	\de\mu_\C\big|_{(\bx,\by)}(\dot\bx,\dot\by)=\dot y_i^{(k)}x_i^{(k)}+y_i^{(k)}\dot x_i^{(k)}-\dot x_i^{(k-1)}y_i^{(k-1)}-x_i^{(k-1)}\dot y_i^{(k-1)}=0
\end{equation}
for all $i=1,\dots,n$ and $k=1,\dots,\ell_i$.  

Let $(\,\delbar_E,\mathcal F(t),\varphi(t))=\mathcal T(\bx(t),\by(t))$.  The residue of the first variation $\dot\varphi$ of the Higgs field at $p_i$ is
\begin{equation}\label{eqn: dot phi residue}
	\dot\varphi_i=\dot x_i^{(\ell_i)}y_i^{(\ell_i)}+x_i^{(\ell_i)}\dot y_i^{(\ell_i)},
\end{equation}
which is already trace-free by iterative application of \eqref{eqn: de mu equation}.  The variation of the flag data is induced by an appropriate $\dot\nu\in C^\infty(C,\End(E))$ which we now construct.  Recall our notation from \eqref{eqn: def of X}
	\[X_i^{(k)}=x_i^{(\ell_i)}\cdots x_i^{(k)}\]
and recall from the construction of $\mathcal T$ that $F_i^{(k)}=\im X_i^{(k)}$.  Since each $x_i^{(k)}$ is injective (\Cref{lemma: the x are injective}) $X_i^{(k)}$ has full rank, and hence $(X_i^{(k)})^\dagger X_i^{(k)}$ is invertible.  We will denote $X_i^{(-k)}=(X_i^{(k)\dagger}X_i^{(k)})^{-1}X_i^{(k)\dagger}$, which is a left inverse of $X_i^{(k)}$ for $k=1,\dots,\ell$.  Set
\begin{equation}\label{eqn: def of B}
	\dot B_i^{(k)}=X_i^{(k+1)}\dot x_i^{(k)}X_i^{(-k)},
\end{equation}
which satisfies the key property
\begin{equation}\label{eqn: B defining property}
	\dot B_i^{(k)}X_i^{(k)}=X_i^{(k+1)}\dot x_i^{(k)}.
\end{equation}
In fact, for any $k_0<k$ we have
	\[\dot B_i^{(k)}X_i^{(k_0)}=X_i^{(k+1)}\dot x_i^{(k)}x_i^{(k-1)\dots(k_0)}\]
where $x_i^{(b)\dots(a)}:=x_i^{(b)}\cdots x_i^{(a)}$ for all $a\leq b$.  Diagrammatically, this is
	\begin{equation*}
        \begin{tikzpicture}[baseline=(current bounding box.center), scale=0.55, >=stealth, thick]
            \draw[gray!60, thin, dashed] (-1.5, -2) -- (9.5, -2);
            \node[left, gray!80, font=\scriptsize] at (-1.5, -2) {$\C^{r_i^{(j)}}$};
            
            \draw[gray!60, thin, dashed] (-1.5, -1) -- (9.5, -1);
            \node[left, gray!80, font=\scriptsize] at (-1.5, -1) {$\C^{r_i^{(k)}}$};
            
            \draw[gray!60, thin, dashed] (-1.5, 0) -- (9.5, 0);
            \node[left, gray!80, font=\scriptsize] at (-1.5, 0) {$\C^{r_i^{(k+1)}}$};
            
            \draw[gray!60, thin, dashed] (-1.5, 2) -- (9.5, 2);
            \node[left, gray!80, font=\scriptsize] at (-1.5, 2) {$\C^{r_\star}$};

            \coordinate (A0) at (9, -2);
            \coordinate (A2) at (5, 2);
            \coordinate (A3) at (2, -1);
            \coordinate (A4) at (1, 0);
            \coordinate (A5) at (-1, 2);

            \draw[->] (A0) -- (A2) node[midway, above right, xshift=2pt, yshift=-4pt, font=\scriptsize] {$X_i^{(j)}$};
            \draw[->] (A2) -- (A3) node[midway, above =6pt, xshift = -2pt, yshift = -2pt, font=\scriptsize] {$X_i^{(-k)}$};
            \draw[->] (A3) -- (A4) node[midway, above = 2pt, xshift = 5pt, font=\scriptsize] {$\dot x_i^{(k)}$};
            \draw[->] (A4) -- (A5) node[midway, above = 2pt, xshift = 10pt, font=\scriptsize] {$X_i^{(k+1)}$};

            \foreach \p in {0,2,3,4,5} {
                \fill (A\p) circle (1.5pt);
            }
        \end{tikzpicture}
        \ = \ 
        \begin{tikzpicture}[baseline=(current bounding box.center), scale=0.55, >=stealth, thick]
            \draw[gray!60, thin, dashed] (0.5, -2) -- (5.5, -2);
            \node[left, gray!80, font=\scriptsize] at (0.5, -2) {$\C^{r_i^{(j)}}$};
            
            \draw[gray!60, thin, dashed] (0.5, -1) -- (5.5, -1);
            \node[left, gray!80, font=\scriptsize] at (0.5, -1) {$\C^{r_i^{(k)}}$};
            
            \draw[gray!60, thin, dashed] (0.5, 0) -- (5.5, 0);
            \node[left, gray!80, font=\scriptsize] at (0.5, 0) {$\C^{r_i^{(k+1)}}$};
            
            \draw[gray!60, thin, dashed] (0.5, 2) -- (5.5, 2);
            \node[left, gray!80, font=\scriptsize] at (0.5, 2) {$\C^{r_\star}$};

            \coordinate (B0) at (5, -2);
            \coordinate (B3) at (4, -1);
            \coordinate (B4) at (3, 0);
            \coordinate (B5) at (1, 2);

            \draw[->] (B0) -- (B3) node[midway, above right, xshift = -2pt, font=\scriptsize] {$x_i^{(k-1)\dots(j)}$};
            \draw[->] (B3) -- (B4) node[midway, above = 2pt, xshift = 5pt, font=\scriptsize] {$\dot x_i^{(k)}$};
            \draw[->] (B4) -- (B5) node[midway, above = 2pt, xshift = 10pt, font=\scriptsize] {$X_i^{(k+1)}$};

            \foreach \p in {0,3,4,5} {
                \fill (B\p) circle (1.5pt);
            }
        \end{tikzpicture}
    \end{equation*}
\begin{proposition}\label{prop: nu induces flag change}
	Let $\dot\nu\in C^\infty(C,\End(E))$ be any section with
\begin{equation}\label{eqn: nu at punctures}
	\dot\nu(p_i)=\sum_{k=1}^{\ell_i}\dot B_i^{(k)}=\sum_{k=1}^{\ell_i}X_i^{(k+1)}\dot x_i^{(k)}X_i^{(-k)}.
\end{equation}
	Then $\dot\nu$ induces the first variation of $\mathcal F(t)$ in the sense of \Cref{def: flag changing endomorphism}.
\end{proposition}
\begin{proof}
	In light of \Cref{lemma: infinitesimal flag change}, we must show that the residue $\dot\Phi=\dot\varphi+[\varphi,\dot\nu]$ is strongly parabolic with respect to $\mathcal F$.  Let $\varphi_i,\dot\varphi_i,\dot\Phi_i$ be the residues of $\varphi,\dot\varphi,\dot\Phi$ and $p_i$.  Using the key property \eqref{eqn: B defining property} of $\dot B_i^{(k)}$ in the case $k=\ell_i$,
	\begin{align}
		\dot\Phi_i
			=(\dot\varphi_i+[\varphi,\dot\nu(p_i)])
				&=\dot x_i^{(\ell_i)}y_i^{(\ell_i)}+x_i^{(\ell_i)}\dot y_i^{(\ell_i)}+\sum_{k=1}^{\ell_i}x_i^{(\ell_i)}y_i^{(\ell_i)}\dot B_i^{(k)}-\dot B_i^{(k)}x_i^{(\ell_i)}y_i^{(\ell_i)}
					\notag\\
				&=x_i^{(\ell_i)}\dot y_i^{(\ell_i)}+x_i^{(\ell_i)}y_i^{(\ell_i)}B_i^{(\ell_i)}+\sum_{k=1}^{\ell_i-1}\left(x_i^{(\ell_i)}y_i^{(\ell_i)}\dot B_i^{(k)}-\dot B_i^{(k)}x_i^{(\ell_i)}y_i^{(\ell_i)}\right).
					\label{eqn: Phi dot expansion}
	\end{align}
	Therefore
	\begin{align}
		\dot\Phi_iX_i^{(k_0)}
			&=x_i^{(\ell_i)}\dot y_i^{(\ell_i)}X_i^{(k_0)}
				+x_i^{(\ell_i)}y_i^{(\ell_i)}\dot x_i^{(\ell_i)}x_i^{(\ell_i-1)\dots(k_0)}	\notag\\
				&\qquad+\sum_{k=1}^{\ell_i-1}\left(x_i^{(\ell_i)}y_i^{(\ell_i)}\dot B_i^{(k)}-\dot B_i^{(k)}x_i^{(\ell_i)}y_i^{(\ell_i)}\right)X_i^{(k_0)}
					\label{eqn: Phi dot X}
	\end{align}
	The idea of this proof is to use the moment map equations from \eqref{eqn: mu C i}
	\begin{equation}\label{eqn: mu equation restated}
		\tau_i^{(k)}\Id=\mu_{\C,i}^{(k)}=y_i^{(k)}x_i^{(k)}-x_i^{(k-1)}y_i^{(k-1)}
	\end{equation}
	and their linearizations \eqref{eqn: de mu equation} to iteratively commute the $\dot x$, $\dot y$ parts in \eqref{eqn: Phi dot X} to the right.  The following claim describes the resulting formula at any given stage.
	\begin{claim}\label{claim: Phi dot X induction}
		For all $k=k_0,\dots,\ell_i$ we have
		\begin{align}
			\dot\Phi_iX_i^{(k_0)}
				&=X_i^{(k)}\dot y_i^{(k)}x_i^{(k)\dots(k_0)}+X_i^{(k)}y_i^{(k)}\dot x_i^{(k)}x_i^{(k-1)\dots(k_0)}	\notag\\
					&\qquad+\sum_{j=1}^{k-1}\left(x_i^{(\ell_i)}y_i^{(\ell_i)}B_i^{(j)}-B_i^{(j)}x_i^{(\ell_i)}y_i^{(\ell_i)}\right)X_i^{(k_0)}
						\label{eqn: Phi dot X induction}
		\end{align}
		(in the case $k=k_0$ the $x_i^{(k-1)\dots,(k_0)}$ is omitted).
	\end{claim}
	\begin{proof}[Proof of \Cref{claim: Phi dot X induction}]
		The $k=\ell_i$ case is \eqref{eqn: Phi dot X}.  Proceeding by induction, we assume the result holds for $k>k_0$, and show it also holds for $k-1$.  Using the linearized complex moment map equation from \eqref{eqn: de mu equation} (relation R1), the key property of $\dot B_i^{(k)}$ \eqref{eqn: B defining property} (relation R2), and the commutation rule \Cref{lemma: xy commutation rule} (relation R3), we calculate the sum of the first two terms above with the $j=k-1$ term from the summation.  In each line, the under-braced parts get replaced by the over-bracketed parts in the next line by using one of these three relations.
		{\allowdisplaybreaks
		\begin{align}
			&X_i^{(k)}\underbrace{\dot y_i^{(k)}x_i^{(k)}}_{R1} x_i^{(k-1)\dots(k_0)}+X_i^{(k)}\underbrace{y_i^{(k)}\dot x_i^{(k)}}_{R1} x_i^{(k-1)\dots(k_0)}	
					\notag\\
				&\qquad+X_i^{(\ell_i)}y_i^{(\ell_i)}\underbrace{B_i^{(k-1)}X_i^{(k_0)}}_{R2}-B_i^{(k-1)}x_i^{(\ell_i)}\underbrace{y_i^{(\ell_i)}X_i^{(k_0)}}_{R3}\\
			&=X_i^{(k)}\overbracket{x_i^{(k-1)}\dot y_i^{(k-1)}}^{R1}x_i^{(k-1)\dots(k_0)}+X_i^{(k)}\overbracket{\dot x_i^{(k-1)}y_i^{(k-1)}}^{R1}x_i^{(k-1)\dots(k_0)}	
					\notag\\
				&\qquad+x_i^{(\ell_i)}\underbrace{y_i^{(\ell_i)}\hspace{4ex}}_{R3}\hspace{-4ex}\overbracket{X_i^{(k)}\dot x_i^{(k-1)}}^{R2}X_i^{(k-2)\dots(k_0)}	
					\notag\\
				&\qquad-\underbrace{B_i^{(k-1)}\hspace{7ex}}_{R2}\hspace{-7ex}\overbracket{X_i^{(k-1)}y_i^{(k-1)}}^{R3}x_i^{(k-1)\dots(k_0)}+\overbracket{\sum_{m=k}^{\ell_i}\tau_i^{(m)}}^{R3}\underbrace{B_i^{(k-1)}X_i^{(k_0)}}_{R2}\\
			&=X_i^{(k-1)}\dot y_i^{(k-1)}x_i^{(k-1)\dots(k_0)}+X_i^{(k)}\dot x_i^{(k-1)}y_i^{(k-1)}x_i^{(k-1)\dots(k_0)}	
					\notag\\
				&\qquad+\overbracket{X_i^{(k-1)}y_i^{(k-1)}}^{R3}\dot x_i^{(k-1)}x_i^{(k-2)\dots(k_0)}-\overbracket{\sum_{m=k}^{\ell_i}\tau_i^{(m)}}^{R3}X_i^{(k)}\dot x_i^{(k)}x_i^{(k-2)\dots(k_0)}	
					\notag\\
				&\qquad-\overbracket{X_i^{(k)}\dot x_i^{(k-1)}}^{R2}y_i^{(k-1)}x_i^{(k-1)\dots(k_0)}+\sum_{m=k}^{\ell_i}\tau_i^{(m)}\overbracket{X_i^{(k)}\dot x_i^{(k)}}^{R2}x_i^{(k-2)\dots(k_0)}\\
			&=X_i^{(k-1)}\dot y_i^{(k-1)}x_i^{(k-1)\dots(k_0)}+X_i^{(k-1)}y_i^{(k-1)}\dot x_i^{(k-1)}x_i^{(k-2)\dots(k_0)}.
		\end{align}
		}The last line is the first two terms in the desired formula \eqref{eqn: Phi dot X induction}.  Since we left the remaining summands alone, this completes the induction step.  The claim follows.
	\end{proof}
	Now we use the claim in the case $k=k_0$ and simplify:
	{\allowdisplaybreaks
	\begin{align}
		\dot\Phi_iX_i^{(k_0)}
			&=X_i^{(k_0)}\underbrace{\dot y_i^{(k_0)}x_i^{(k_0)}}_{R1}+X_i^{(k_0)}\underbrace{y_i^{(k_0)}\dot x_i^{(k_0)}}_{R1}
					\notag\\
				&\qquad+\sum_{j=1}^{k_0-1}\left(x_i^{(\ell_i)}y_i^{(\ell_i)}B_i^{(j)}X_i^{(k_0)}-B_i^{(j)}x_i^{(\ell_i)}\underbrace{y_i^{(\ell_i)}X_i^{(k_0)}}_{R3}\right)	\notag\\
			&=X_i^{(k_0)}\overbracket{x_i^{(k_0-1)}\dot y_i^{(k_0-1)}}^{R1}+X_i^{(k_0)}\overbracket{\dot x_i^{(k_0-1)}y_i^{(k_0-1)}}^{R1}
					\notag\\
				&\qquad+\sum_{j=1}^{k_0-1}\Bigg(x_i^{(\ell_i)}\underbrace{y_i^{(\ell_i)}X_i^{(j+1)}}_{R3}\dot x_i^{(j)}X_i^{(-j)}X_i^{(k_0)}-B_i^{(j)}\overbracket{X_i^{(k_0-1)}y_i^{(k_0-1)}}^{R3}
					\notag\\
				&\qquad\qquad\qquad-\overbracket{\sum_{m=k_0}^{\ell_i}\tau_i^{(m)}}^{R3}B_i^{(j)}X_i^{(k_0)}\Bigg)
					\notag\\
			&=X_i^{(k_0-1)}\dot y_i^{(k_0-1)}+X_i^{(k_0)}\dot x_i^{(k_0-1)}y_i^{(k_0-1)}
					\notag\\
				&\qquad+\sum_{j=1}^{k_0-1}\Bigg(\overbracket{X_i^{(k_0-1)}y_i^{(k_0-1)}}^{R3}x_i^{(k_0-1)\dots(j+1)}\dot x_i^{(j)}X_i^{(-j)}X_i^{(k_0)}
					\notag\\
				&\qquad\qquad\qquad+\overbracket{\sum_{m=k_0}^{\ell_i}\tau_i^{(m)}}^{R3}X^{(j+1)}\dot x_i^{(j)}X_i^{(-j)}X_i^{(k_0)}
					-\sum_{m=k_0}^{\ell_i}\tau_i^{(m)}B_i^{(j)}X_i^{(k_0)}\Bigg)
					\notag\\
				&\qquad+\sum_{j=1}^{k_0-1}B_j^{(j)}X_i^{(k_0-1)}y_i^{(k_0-1)}
					\notag\\
			&=X_i^{(k_0-1)}\dot y_i^{(k_0-1)}
				+\sum_{j=1}^{k_0-1}X_i^{(k_0-1)}y_i^{(k_0-1)}x_i^{(k_0-1)\dots(j+1)}\dot x_i^{(j)}X_i^{(-j)}X_i^{(k_0)}
					\notag\\
				&\qquad+\sum_{j=1}^{k_0-2}B_i^{(j)}X_i^{(k_0-1)}y_i^{(k_0-1)}
	\end{align}
	}where the last equality uses relation R2 and the definition of $\dot B_i^{(k)}$ \eqref{eqn: def of B} to cancel terms.  Because $B_i^{(j)}$ factors through $X_i^{(k_0-1)}$ for $j\leq k_0-2$, this shows $\im(\dot\Phi_i X_i^{(k_0)})\subseteq\im(X_i^{(k_0-1)})=F_i^{(k_0-1)}$.\footnote{
		In the case $k_0=1$, our convention to take $x_i^{(0)}=0$ and $y_i^{(0)}=0$ shows $\dot\Phi_iX_i^{(1)}=0$.  Doing this is consistent with the correct interpretation of \eqref{eqn: mu C i} for $k=0$.
	}
\end{proof}

\subsection{Strongly Parabolic Case via Tautological 1-Forms}\label{subsec: strongly parabolic case via tautological 1-form}
Recall that in the strongly parabolic case (where $\tau,\sigma=0$), the holomorphic symplectic forms $\Omega_\mathrm{quiv},\Omega_\mathrm{Higgs}$ are the exterior derivatives of the tautological 1-forms $L_\mathrm{quiv},L_\mathrm{Higgs}$, respectively.  Hence, we can reduce \Cref{thm: holo symplectic forms agree} to the following.
\begin{theorem}\label{thm: tautological 1-forms agree}
	In the strongly parabolic case $\sigma,\tau=0$, we have $\mathcal T^*L_\mathrm{Higgs}=2\pi\,L_\mathrm{quiv}$.
\end{theorem}
\begin{proof}
	Let $(\bx,\by)\in\mu_\C^{-1}(\tau)$ and let $(\dot\bx,\dot\by)$ be a unitary deformation.  Set $(\,\delbar_E,\mathcal F,\varphi)=\mathcal T(\bx,\by)$ and $(\dot\nu,\dot\varphi)=\de\mathcal T(\dot\bx,\dot\by)$.  Applying \Cref{prop: Omega on holo deformation} and Stokes' theorem,
	\begin{align}
		\mathcal T^*L_\mathrm{Higgs}\big|_{(\bx,\by)}(\dot x,\dot y)
			&=L_\mathrm{Higgs}\big|_{(\,\delbar,\mathcal F,\varphi,h)}(\dot A,\dot\Phi)
			\notag\\
			&=-\mathbbm i\int_{\P^1}\tr(\,\delbar\dot\nu\wedge\varphi)
			\notag\\
			&=-\mathbbm i\sum_{i=1}^n\lim_{\delta\to0}\oint_{\partial B_\delta(p_i)}\tr(\dot\nu\varphi)
			\notag\\
			&=2\pi\sum_{i=1}^n\lim_{\delta\to0}\tr(\dot\nu(p_i)\varphi_i).
	\end{align}
	The theorem follows from the next lemma and the expression \eqref{eqn: quiver tautological 1 form} for $L_\mathrm{quiv}$.
\end{proof}

\begin{lemma}\label{lemma: local tautological pairing}
	For all $i=1,\dots,n$, we have
	\begin{equation}\label{eqn: taut 1-form residue expression}
		\tr(\dot\nu(p_i)\varphi_i)=\sum_{k=1}^{\ell_i}\tr(\dot x_i^{(k)}y_i^{(k)}).
	\end{equation}
\end{lemma}

\begin{proof}[Proof of \Cref{lemma: local tautological pairing}]
	We use the cyclic property of trace and \Cref{lemma: xy commutation rule} to get
	\begin{align}
		\tr(\dot B_i^{(k)}x_i^{(\ell_i)}y_i^{(\ell_i)})
			&=\tr(X_i^{(k+1)}\dot x_i^{(k)}X_i^{(-k)}x_i^{(\ell_i)}y_i^{(\ell_i)})
			\notag\\
			&=\tr(\dot x_i^{(k)}X_i^{(-k)}x_i^{(\ell_i)}y_i^{(\ell_i)}X_i^{(k+1)})
			\notag\\
			&=\tr(\dot x_i^{(k)}X_i^{(-k)}x_i^{(\ell_i)}x_i^{(\ell_i-1)\dots(k)}y_i^{(k)})
				\notag\\
			&=\tr(\dot x_i^{(k)}y_i^{(k)})
	\end{align}
	See \Cref{fig: taut 1 form calculation} for a diagramatic calculation---diagrams of this kind were instrumental in the exploration that lead to our proofs in this section.
\end{proof}

\begin{figure}[h]
	\begin{align*}
		\begin{tikzpicture}[baseline=(current bounding box.center), scale=0.5, >=stealth, thick]
			\coordinate (L0) at (5, 2);
			\coordinate (L1) at (4, 1);
			\coordinate (L2) at (3, 2);
			\coordinate (L3) at (2, 1);
			\coordinate (L4) at (1, 0);
			\coordinate (L5) at (0, -1);
			\coordinate (L6) at (-1, 0);
			\coordinate (L7) at (-2, 1);
			\coordinate (L8) at (-3, 2);
			\draw[->] (L0) -- (L1) node[midway, above = 8pt, xshift = 6pt, font=\scriptsize] {$y_i^{(\ell_i)}$};
			\draw[->] (L1) -- (L2) node[midway, above = 8pt, xshift = 2pt, font=\scriptsize] {$x_i^{(\ell_i)}$};
			\draw[->] (L2) -- (L5) node[midway, below right, xshift = -6pt, font=\scriptsize] {$X_i^{(-k)}$};
			\draw[->] (L5) -- (L6) node[midway, above = 4pt, xshift = 6pt, font=\scriptsize] {$\dot x_i^{(k)}$};
			\draw[->] (L6) -- (L8) node[midway, above right, xshift = -3pt, font=\scriptsize] {$X_i^{(k+1)}$};
			\draw[->, dashed, rounded corners=15pt] (L8) -- (-3, -2.5) -- node[above, font=\scriptsize] {$\tr$} (5, -2.5) -- (L0);
			\foreach \p in {0,1,2,5,6,7,8} {
				\fill (L\p) circle (1.5pt);
			}
		\end{tikzpicture}
		&\ =\ 
			\begin{tikzpicture}[baseline=(current bounding box.center), scale=0.5, >=stealth, thick]
			\coordinate (M0) at (5, 0);
			\coordinate (M1) at (4, 1);
			\coordinate (M1b) at (4,-1);
			\coordinate (M2) at (3, 2);
			\coordinate (M2b) at (3, 0);
			\coordinate (M3) at (2, 1);
			\coordinate (M4) at (1, 2);
			\coordinate (M5) at (0, 1);
			\coordinate (M6) at (-1, 0);
			\coordinate (M7) at (-2, -1);
			\coordinate (M8) at (-3, -0);
			\draw[->] (M0) -- (M2) node[midway, above = 4pt, xshift=8pt, font=\scriptsize] {$X_i^{(k+1)}$};
			\draw[->] (M2) -- (M3) node[midway, above = 8pt, xshift = 6pt, font=\scriptsize] {$y_i^{(\ell_i)}$};
			\draw[-, dotted] (M0) -- (M1b);
			\draw[-, dotted] (M1b) -- (M2b);
			\node at (4, 0) {$\Downarrow$};
			\draw[-, dotted] (M1) -- (M2b);
			\draw[-, dotted] (M2b) -- (M3);
			\node at (3, 1) {$\Downarrow$};
			\draw[->] (M3) -- (M4) node[midway, above = 8pt, xshift = 2pt, font=\scriptsize] {$x_i^{(\ell_i)}$};
			\draw[->] (M4) -- (M7) node[midway, below right, xshift = -6pt, font=\scriptsize] {$X_i^{(-k)}$};
			\draw[->] (M7) -- (M8) node[midway, above = 4pt, xshift=6pt, font=\scriptsize] {$\dot x_i^{(k)}$};
			\foreach \p in {0,2,3,4,7,8} {
				\fill (M\p) circle (1.5pt);
			}
			\fill (M1b) circle (1.5pt);
			\fill (M2b) circle (1.5pt);
		\end{tikzpicture}\\
		&\ =\ 
			\begin{tikzpicture}[baseline=(current bounding box.center), scale=0.5, >=stealth, thick]
			\coordinate (R0) at (5, 0);
			\coordinate (R1) at (4, -1);
			\coordinate (R2) at (3, 0);
			\coordinate (R3) at (2, 1);
			\coordinate (R4) at (1, 2);
			\coordinate (R5) at (0, 1);
			\coordinate (R6) at (-1, 0);
			\coordinate (R7) at (-2, -1);
			\coordinate (R8) at (-3, 0);
			\draw[->] (R0) -- (R1) node[midway, above = 1pt, xshift = -1pt, font=\scriptsize] {$y_i^{(k)}$};
			\draw[->] (R1) -- (R4) node[midway, above = 4pt, xshift = 5pt, font=\scriptsize] {$X_i^{(k)}$};
			\draw[->] (R4) -- (R7) node[midway, below right, xshift=-6pt, font=\scriptsize] {$X_i^{(-k)}$};
			\draw[->] (R7) -- (R8) node[midway, above = 4pt, xshift=6pt, font=\scriptsize] {$\dot x_i^{(k)}$};
			\foreach \p in {0,1,4,7,8} {
				\fill (R\p) circle (1.5pt);
			}
		\end{tikzpicture}\\[10pt]
		&\ =\ 
		\begin{tikzpicture}[baseline=(current bounding box.center), scale=0.5, >=stealth, thick]
			\coordinate (R0) at (-1, 0);
			\coordinate (R1) at (-2, -1);
			\coordinate (R2) at (-3, 0);
			\draw[->] (R0) -- (R1) node[midway, above = 8pt, xshift = 6pt, font=\scriptsize] {$y_i^{(k)}$};
			\draw[->] (R1) -- (R2) node[midway, above = 8pt, xshift = 2pt, font=\scriptsize] {$\dot x_i^{(k)}$};
			\foreach \p in {0,1,2} {
				\fill (R\p) circle (1.5pt);
			}
		\end{tikzpicture}
	\end{align*}
	\caption{Diagramatic calculation of $\tr(\dot\nu(p_i)\varphi_i)$.  Equalities hold under the trace operation.}
	\label{fig: taut 1 form calculation}
\end{figure}
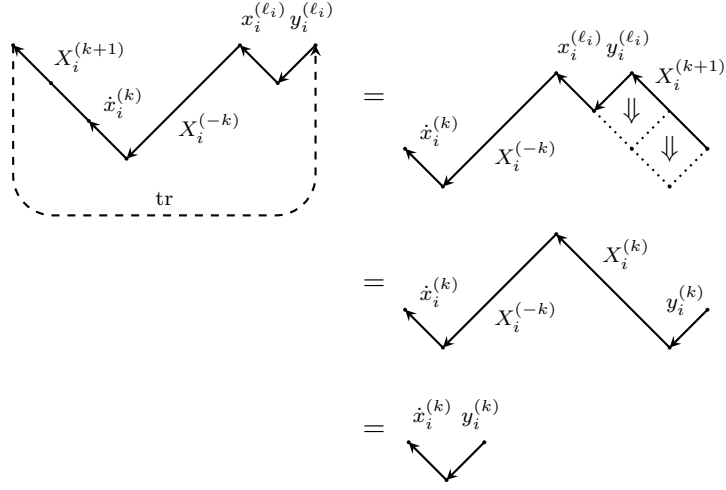

\subsection{Weakly Parabolic Case}\label{subsec: weakly parabolic case}
Now we prove \Cref{thm: holo symplectic forms agree} in the general (weakly parabolic) setting, where the tautological 1-form is unavailable.

\begin{proof}[Proof of \Cref{thm: holo symplectic forms agree}]
	Let $(\bx,\by)\in\mu_\C^{-1}(\tau)\cap\mu_\R^{-1}(\beta)$, and let $(\dot\bx_1,\dot\by_1)$ and $(\dot\bx_2,\dot\by_2)$ be unitary deformations of $(\bx,\by)$.  Set $(\,\delbar_E,\mathcal F,\varphi)=\mathcal T(\bx,\by)$ and consider the holomorphic deformations $(\dot\nu_j,\dot\varphi_j)$ representing $\de\mathcal T(\dot\bx_j,\dot\by_j)$ described in \Cref{prop: nu induces flag change} for $j=1,2$.  Recall from \eqref{eqn: quiver holo symp form}
		\[\Omega_\mathrm{quiv}\big|_{(\bx,\by)}((\dot\bx_1,\dot\by_1),(\dot\bx_2,\dot\by_2))=\tr(\dot\bx_1\dot\by_2-\dot\bx_2\dot\by_1)=\sum_{i,k}\tr\left(\dot x_{1,i}^{(k)}\dot y_{2,i}^{(k)}-\dot x_{2,i}^{(k)}\dot y_{1,i}^{(k)}\right).\]
	
	On the other hand, successively applying \Cref{prop: Omega on holo deformation}, the holomorphicity of $\varphi$ and $\dot\varphi_j$, the cyclic property of trace, and Stokes' theorem,
	\begin{align}
		\Omega_\mathrm{Higgs}
			&=\Omega_\mathrm{Higgs}((\dot\nu_1,\dot\phi_1),(\dot\nu_2,\dot\phi_2))
				\notag\\
			&=-\mathbbm i\int\tr\left(\,\delbar\dot\nu_1\wedge\dot(\dot\varphi_2+[\varphi,\dot\nu_2])-\delbar\dot\nu_2\wedge\dot(\varphi_1+[\varphi,\dot\nu_1])\right)
				\notag\\
			&=-\mathbbm i\int\tr\left(\,\delbar\dot\nu_1\wedge\dot\varphi_2-\delbar\dot\nu_2\wedge\dot\varphi_1+\delbar\dot\nu_1\wedge[\varphi,\dot\nu_2]-\delbar\dot\nu_2\wedge[\varphi,\dot\nu_1]\right)
				\notag\\
			&=-\mathbbm i\int\tr\left(\,\delbar\dot\nu_1\wedge\dot\varphi_2-\delbar\dot\nu_2\wedge\dot\varphi_1+\delbar\dot\nu_1\wedge\varphi\dot\nu_2-\delbar\dot\nu_1\wedge\dot\nu_2\varphi-\delbar\dot\nu_2\wedge\varphi\dot\nu_1+\delbar\dot\nu_2\wedge\dot\nu_1\varphi\right)
				\notag\\
			&=-\mathbbm i\int\tr\left(\,\delbar(\dot\nu_1\dot\varphi_2-\dot\nu_2\dot\varphi_1)+\dot\nu_2\delbar\dot\nu_1\wedge\varphi-\delbar\dot\nu_1\wedge\dot\nu_2\varphi-\dot\nu_1\delbar\dot\nu_2\wedge\varphi+\delbar\dot\nu_2\wedge\dot\nu_1\varphi\right)
				\notag\\
			&=-\mathbbm i\int\tr\left(\,\delbar(\dot\nu_1\dot\varphi_2-\dot\nu_2\dot\varphi_1)+\delbar(\dot\nu_2\dot\nu_1-\dot\nu_1\dot\nu_2)\wedge\varphi\right)
				\notag\\
			&=2\pi\sum_{i=1}^n\tr\left(\dot\nu_1(p_i)\dot\varphi_{2,i}-\dot\nu_2(p_i)\dot\varphi_{1,i}+[\dot\nu_2(p_i),\dot\nu_1(p_i)]\varphi_i\right)
				\label{eqn: Stokes on holo symplectic form}
	\end{align}
	where $\dot\varphi_{j,i}=\mathrm{res}_{p_i}\dot\varphi_j$.  The theorem now follows from the next lemma.\end{proof}

\begin{lemma}\label{lemma: Omega at single puncture}
	For all $i=1,\dots,n$, we have
	\begin{equation}
		T_i:=\tr\left(\dot\nu_1(p_i)\dot\varphi_{2,i}-\dot\nu_2(p_i)\dot\varphi_{1,i}+[\dot\nu_2(p_i),\dot\nu_1(p_i)]\varphi_i\right)=\sum_{k=1}^{\ell_i}\Omega_i^{(k)}
	\end{equation}
	where $\Omega_i^{(k)}=\tr(\dot x_{1,i}^{(k)}\dot y_{2,i}^{(k)}-\dot x_{2,i}^{(k)}\dot y_{1,i}^{(k)})$.
\end{lemma}
\begin{proof}
	We will fix the index $i$ for the following calculation, so we shall drop this subscripts on the $x$'s, $y$'s, and $\ell$'s, writing $\varphi_i=x^\ell y^\ell$ and $\dot\varphi_{j,i}=\dot x_j^\ell y^\ell+x^\ell\dot y_j^\ell$.  We define $\dot B_j^{(k)}$ as in \eqref{eqn: def of B}, where the index $j$ now distinguishes the two deformations rather than the points $p_i\in D$.

	The following calculation is long and tedious.  We suggest carefully reading the proof of \Cref{prop: nu induces flag change} and recalling the three relations R1, R2, R3 before reading this proof.

	Recall $\dot\nu_j(p_i)=\sum_{k=1}^{\ell}\dot B_j^{(k)}$ and $\dot\varphi_i=\dot x_i^{(\ell_i)}y_i^{(\ell_i)}+x_i^{(\ell_i)}\dot y_i^{(\ell_i)}$ from \eqref{eqn: dot phi residue}.  Since $\dot B_j^{(\ell)}x^{(\ell)}=\dot x_j^{(\ell)}$, 
	{\allowdisplaybreaks
	\begin{align}
		T_i
			&=\tr\Bigg(\sum_{k=1}^{\ell}\dot B_1^{(k)}x^{(\ell)}\dot y_2^{(\ell)}-\dot B_2^{(k)}x^{(\ell)}\dot y_1^{(\ell)}
					+\sum_{k=1}^{\ell}\dot B_1^{(k)}\dot x_2^{(\ell)}y^{(\ell)}-\dot B_2^{(k)}\dot x_1^{(\ell)}y^{(\ell)}
					\notag\\
				&\qquad\qquad+\sum_{k,m=1}^{\ell}\dot B_2^{(k)}\dot B_1^{(m)}x^{(\ell)}y^{(\ell)}-\dot B_1^{(k)}\dot B_2^{(m)}x^{(\ell)}y^{(\ell)}\Bigg)
					\notag\\
			&=\tr\Bigg(\Omega_i^{(\ell_i)}
					+\sum_{k=1}^{\ell-1}\dot B_1^{(k)}x^{(\ell)}\dot y_2^{(\ell)}-\dot B_2^{(k)}x^{(\ell)}\dot y_1^{(\ell)}
					\notag\\
				&\qquad\qquad+\sum_{k=1}^{\ell}\sum_{m=1}^{\ell-1}\dot B_2^{(k)}\dot B_1^{(m)}x^{(\ell)}y^{(\ell)}-\dot B_1^{(k)}\dot B_2^{(m)}x^{(\ell)}y^{(\ell)}\Bigg).
					\label{eqn: T_i first expansion}
	\end{align}
	}
	Like in the proof of \Cref{prop: nu induces flag change}, our strategy will be to iteratively apply the relations R1 (linearized moment map equation \eqref{eqn: de mu equation}), R2 (key property of $B_i^{(k)}$ \eqref{eqn: B defining property}), and R3 (moment map equation \eqref{eqn: mu equation restated}) to pull the subsequent terms $\Omega_i^{(k)}$ out of the sum.  This time, however, we also need to reformulate the remaining summands at each step.
	\begin{claim}\label{claim: T_i induction expression}
		For all $k_0=1,\dots,\ell$,
		\begin{align}
			T_i&=\sum_{k=k_0}^\ell\Omega_i^{(k)}
					\notag\\
				&+\sum_{m=1}^{k_0-1}\underbrace{\tr\left(X^{(-(k_0+1))}\dot B_1^{(m)}X^{(k_0)}\dot y_2^{(k_0)}-X^{(-(k_0+1))}\dot B_2^{(m)}X^{(k_0)}\dot y_1^{(k_0)}\right)}_{R_m^{(k_0)}}
					\notag\\
				&+\sum_{k=1}^{k_0}\sum_{m=1}^{k_0-1}\underbrace{\tr\left(X^{(-(k_0+1))}\dot B_2^{(k)}\dot B_1^{(m)}X^{(k_0)}y^{(k_0)}-X^{-(k_0+1)}\dot B_1^{(k)}\dot B_2^{(m)}X^{(k_0)}y^{(k_0)}\right)}_{S_{k,m}^{(k_0)}}
					\label{eqn: T induction formula}
		\end{align}
		where $X^{-(k_0+1)}$ is omitted when $k_0=\ell$.
	\end{claim}
	\begin{proof}[Proof of \Cref{claim: T_i induction expression}]
		The $k_0=\ell$ case is \eqref{eqn: T_i first expansion}.  Assuming the claim holds for $k_0\in\{2,\dots,\ell\}$, we will show it also holds for $k_0-1$.  We denote the summands $R_m^{(k_0)}$ and $S_{k,m}^{(k_0)}$ as shown above.  Note that $X^{(-(k_0+1))}\dot B_j^{(k_0)}=\dot x_j^{(k_0)}X^{(-k_0)}$.  Thus
		\begin{align}
			S_{k_0,m}^{(k_0)}
				&=\tr\left(\dot x_2^{(k_0)}X^{(-k_0)}\dot B_1^{(m)}X^{(k_0)}y^{(k_0)}-\dot x_1^{(k_0)}X^{(-k_0)}\dot B_2^{(m)}X^{(k_0)}y^{(k_0)}\right)
						\notag\\
				&=\tr\left(X^{(-k_0)}\dot B_1^{(m)}X^{(k_0)}y^{(k_0)}\dot x_2^{(k_0)}-X^{(-k_0)}\dot B_2^{(m)}X^{(k_0)}y^{(k_0)}\dot x_1^{(k_0)}\right)
		\end{align}
		Meanwhile, since
		\begin{equation}\label{eqn: peel out x^k_0}
			X^{(-(k_0+1))}\dot B_j^{(m)}=x^{(k_0)\dots(m+1)}\dot x_j^{(m)}X^{(-m)}=x^{(k_0)}X^{(-(k_0))}\dot B_j^{(m)}
		\end{equation}
		for $m<k_0$, we have
		\begin{align}
			R_m^{(k_0)}
				&=\tr\left(X^{(-(k_0+1))}\dot B_1^{(m)}X^{(k_0)}\dot y_2^{(k_0)}-X^{(-(k_0+1))}\dot B_2^{(m)}X^{(k_0)}\dot y_1^{(k_0)}\right)
						\notag\\
				&=\tr\Big(X^{(-k_0)}\dot B_1^{(m)}X^{(k_0)}\dot y_2^{(k_0)}x^{(k_0)}
					-X^{(-k_0)}\dot B_2^{(m)}X^{(k_0)}\dot y_1^{(k_0)}x^{(k_0)}\Big).
		\end{align}
		Applying the relation R1 to the sum $S_{k_0,m}^{(k_0)}+R_m^{(k_0)}$ and regrouping terms,
		\begin{align}
			S_{k_0,m}^{(k_0)}+R_m^{(k_0)}
				&=\tr\Big(X^{(-k_0)}\dot B_1^{(m)}X^{(k_0)}\overbracket{\left(x^{(k_0-1)}\dot y_2^{(k_0-1)}+\dot x_2^{(k_0-1)}y^{(k_0-1)}\right)}^{R1}
						\notag\\
					&\qquad\qquad-X^{(-k_0)}\dot B_2^{(m)}X^{(k_0)}\overbracket{\left(x^{(k_0-1)}\dot y_1^{(k_0-1)}+\dot x_1^{(k_0-1)}y^{(k_0-1)}\right)}^{R1}\Big)
						\notag\\
				&=\tr\Big(X^{(-k_0)}\dot B_1^{(m)}X^{(k_0-1)}\dot y_2^{(k_0-1)}
						-X^{(-k_0)}\dot B_2^{(m)}X^{(k_0-1)}\dot y_1^{(k_0-1)}
						\notag\\
					&\qquad\qquad+X^{(-k_0)}\dot B_1^{(m)}X^{(k_0)}\dot x_2^{(k_0-1)}y^{(k_0-1)}
						-X^{(-k_0)}\dot B_2^{(m)}X^{(k_0)}\dot x_1^{(k_0-1)}y^{(k_0-1)}\Big).
						\label{eqn: S + R}
		\end{align}
		When $m=k_0-1$,
		\begin{equation}
			X^{(-k_0)}\dot B_j^{(m)}X^{(k_0-1)}=X^{(-k_0)}X^{(k_0)}\dot x^{(k_0-1)}X^{(-(k_0-1))}X^{(k_0-1)}=\dot x^{(k_0-1)}.
		\end{equation}
		Therefore the first two terms collapse to
			\[\tr(\dot x_1^{(k_0-1)}\dot y_2^{(k_0-2)}-\dot x_2^{(k_0-1)}\dot y_1^{(k_0-1)})=\Omega_i^{(k_0-1)},\]
		and the second two terms cancel with $S_{k_0-1,k_0-1}^{(k_0)}$; for using \eqref{eqn: peel out x^k_0}, the cyclic property of trace, and the relation R3,
		{\allowdisplaybreaks
		\begin{align}
			S_{k_0-1,k_0-1}^{(k_0)}
				&=\tr\Big(X^{(-(k_0+1))}\dot B_2^{(k_0-1)}\dot B_1^{(k_0-1)}X^{(k_0)}y^{(k_0)}
						\notag\\
					&\qquad\qquad-X^{(-(k_0+1))}\dot B_1^{(k_0-1)}\dot B_2^{(k_0-1)}X^{(k_0)}y^{(k_0)}\Big)
						\notag\\
				&=\tr\Big(X^{(-k_0)}\dot B_2^{(k_0-1)}\dot B_1^{(k_0-1)}X^{(k_0)}\underbrace{y^{(k_0)}x^{(k_0)}}_{R3}
						\notag\\
					&\qquad\qquad-X^{(-k_0)}\dot B_1^{(k_0-1)}\dot B_2^{(k_0-1)}X^{(k_0)}\underbrace{y^{(k_0)}x^{(k_0)}}_{R3}\Big)
						\notag\\
				&=\tr\Big(X^{(-k_0)}\dot B_2^{(k_0-1)}\underbrace{\dot B_1^{(k_0-1)}\hspace{7ex}}_{R2}\hspace{-7ex}\overbracket{X^{(k_0-1)}y^{(k_0-1)}}^{R3}
						\notag\\
					&\qquad\qquad-X^{(-k_0)}\dot B_1^{(k_0-1)}\underbrace{\dot B_2^{(k_0-1)}\hspace{7ex}}_{R2}\hspace{-7ex}\overbracket{X^{(k_0-1)}y^{(k_0-1)}}^{R3}\Big)
						\notag\\
					&\qquad\qquad+\overbracket{\tau^{(k_0)}}^{R3}\left(X^{(-k_0)}\dot B_2^{(k_0-1)}\dot B_1^{(k_0-1)}X^{(k_0)}-X^{(-k_0)}\dot B_1^{(k_0-1)}\dot B_2^{(k_0-1)}X^{(k_0)}\right)
						\notag\\
				&=\tr\Big(X^{(-k_0)}\dot B_2^{(k_0-1)}\overbracket{X^{(k_0)}\dot x_1^{(k_0-1)}}^{R2}y^{(k_0-1)}
						\notag\\
					&\qquad\qquad-X^{(-k_0)}\dot B_1^{(k_0-1)}\overbracket{X^{(k_0)}\dot x_2^{(k_0-1)}}^{R2}y^{(k_0-1)}\Big),
		\end{align}
		}
		where the term involving $\tau^{(k_0)}$ disappeared because
		\begin{align}
			\tr\left(X^{(-k_0)}\dot B_2^{(k_0-1)}\dot B_1^{(k_0-1)}X^{(k_0)}\right)
				&=\tr\left(X^{(k_0)}X^{(-k_0)}\dot B_2^{(k_0-1)}\dot B_1^{(k_0-1)}\right)
					\notag\\
				&=\tr\left(X^{(k_0)}\dot x_2^{(k_0-1)}X^{(-(k_0-1))}\dot B_1^{(k_0-1)}\right)
					\notag\\
				&=\tr\left(\dot B_1^{(k_0-1)}\dot B_2^{(k_0-1)}\right)
					\label{eqn: trace B1 B2}
		\end{align}
		and similarly
			\[\tr\left(X^{(-k_0)}\dot B_1^{(k_0-1)}\dot B_2^{(k_0-1)}X^{(k_0)}\right)=\tr\left(\dot B_1^{(k_0-1)}\dot B_2^{(k_0-1)}\right),\]
		so the terms cancel.

		To recap, we have shown
		\begin{equation}\label{eqn: R + S + S = Omega}
			R_{k_0-1}^{(k_0)}+S_{k_0,k_0-1}^{(k_0)}+S_{k_0-1,k_0-1}^{(k_0)}=\Omega_i^{(k_0-1)}.
		\end{equation}
		The remaining terms in the right hand side of the induction hypothesis \eqref{eqn: T induction formula} are
		\begin{equation}
			\sum_{m=1}^{k_0-2}\left(R_m^{(k_0)}+S_{k_0,m}^{(k_0)}\right)+\sum_{(k,m)\in I}S_{k,m}^{(k_0)}
		\end{equation}
		where $I=\{1,\dots,k_0-1\}^2\sm\{(k_0-1,k_0-1)\}$.  To complete the induction, we must decrease the upper indices to $k_0-1$.  Continuing from \eqref{eqn: S + R} with $m\leq k_0-2$ and referencing the definitions of $R$ and $S$ in \eqref{eqn: T induction formula},
		{\allowdisplaybreaks
		\begin{align}
			R_m^{(k_0)}+S_{k_0,m}^{(k_0)}
				&=\tr\Big(X^{(-k_0)}\dot B_1^{(m)}X^{(k_0-1)}\dot y_2^{(k_0-1)}
						-X^{(-k_0)}\dot B_2^{(m)}X^{(k_0-1)}\dot y_1^{(k_0-1)}
						\notag\\
					&\qquad\qquad+X^{(-k_0)}\dot B_1^{(m)}X^{(k_0)}\dot x_2^{(k_0-1)}y^{(k_0-1)}
					-X^{(-k_0)}\dot B_2^{(m)}X^{(k_0)}\dot x_1^{(k_0-1)}y^{(k_0-1)}\Big)
						\notag\\
				&=R_m^{(k_0-1)}
						\notag\\
					&\qquad+\tr\Big(X^{(-k_0)}\dot B_1^{(m)}\underbrace{X^{(k_0)}\dot x_2^{(k_0-1)}X^{(-(k_0-1))}}_{\eqref{eqn: def of B}}X^{(k_0-1)}y^{(k_0-1)}
						\notag\\
					&\qquad\qquad-X^{(-k_0)}\dot B_2^{(m)}\underbrace{X^{(k_0)}\dot x_1^{(k_0-1)}X^{(-(k_0-1))}}_{\eqref{eqn: def of B}}X^{(k_0-1)}y^{(k_0-1)}\Big)
						\notag\\
				&=R_m^{(k_0-1)}+\tr\Big(X^{(-k_0)}\dot B_1^{(m)}\overbracket{\dot B_2^{(k_0-1)}}^{\eqref{eqn: def of B}}X^{(k_0-1)}y^{(k_0-1)}
						\notag\\
					&\qquad\qquad\qquad\qquad-X^{(-k_0)}\dot B_2^{(m)}\overbracket{\dot B_1^{(k_0-1)}}^{\eqref{eqn: def of B}}X^{(k_0-1)}y^{(k_0-1)}\Big)
						\notag\\
				&=R_m^{(k_0-1)}-S_{m,k_0-1}^{(k_0-1)}
						\label{eqn: R + S reduction}
		\end{align}
		}
		(note $\eqref{eqn: def of B}$ is the definition of $\dot B_i^{(k)}$).  Finally, for all $k\leq k_0-1$ and $m\leq k_0-2$,
		\begin{align}
			S_{k,m}^{(k_0)}
				&=\tr\left(X^{(-(k_0+1))}\dot B_2^{(k)}\dot B_1^{(m)}X^{(k_0)}y^{(k_0)}-X^{(-(k_0+1))}\dot B_1^{(k)}\dot B_2^{(m)}X^{(k_0)}y^{(k_0)}\right)
						\notag\\
				&=\tr\Big(X^{(-k_0)}\dot B_2^{(k)}\dot B_1^{(m)}X^{(k_0)}\underbrace{y^{(k_0)}x^{(k_0)}}_{R3}
					-X^{(-k_0)}\dot B_1^{(k)}\dot B_2^{(m)}X^{(k_0)}\underbrace{y^{(k_0)}x^{(k_0)}}_{R3}\Big)
						\notag\\
				&=\tr\Big(X^{(-k_0)}\dot B_2^{(k)}\dot B_1^{(m)}\overbracket{X^{(k_0-1)}y^{(k_0-1)}}^{R3}
					-X^{(-k_0)}\dot B_1^{(k)}\dot B_2^{(m)}\overbracket{X^{(k_0-1)}y^{(k_0-1)}}^{R3}\Big)
						\notag\\
					&\qquad+\overbracket{\tau^{(k_0)}}^{R3}\tr\Big(X^{(-k_0)}\dot B_2^{(k)}\dot B_1^{(m)}X^{(k_0)}
						-X^{(-k_0)}\dot B_1^{(k)}\dot B_2^{(m)}X^{(k_0)}\Big)
						\notag\\
				&=S_{k,m}^{(k_0-1)}+\tau^{(k_0)}\tr\left(\dot B_2^{(k)}\dot B_1^{(m)}-\dot B_1^{(k)}\dot B_2^{(m)}\right).
						\label{eqn: S reduction}
		\end{align}
		Where the last equality comes from an identical argument as in \eqref{eqn: trace B1 B2}.  Under the summation over all $(k,m)\in I$ the terms involving $\tau^{(k_0)}$ cancel.  All told, the inductive hypothesis can be reduced as follows:
		\begin{align}
			T_i
				&=\sum_{k=k_0}^{\ell}\Omega_i^{(k)}+\sum_{m=1}^{k_0-1}R_m^{(k_0)}+\sum_{k=1}^{k_0}\sum_{m=1}^{k_0-1}S_{k,m}^{(k_0)}
					&(\text{induction hypothesis})
						\notag\\
				&=\sum_{k=k_0-1}^{\ell}\Omega_i^{(k)}+\sum_{m=1}^{k_0-2}\left(R_m^{(k_0)}+S_{k_0,m}^{(k_0)}\right)+\sum_{(k,m)\in I}S_{k,m}^{(k_0)}
					&(\text{by \eqref{eqn: R + S + S = Omega}})
						\notag\\
				&=\sum_{k=k_0-1}^{\ell}\!\Omega_i^{(k)}+\sum_{m=1}^{k_0-2}\left(R_m^{(k_0-1)}-S_{m,k_0-1}^{(k_0-1)}\right)+\sum_{(k,m)\in I}S_{k,m}^{(k_0-1)}
					&(\text{by \eqref{eqn: R + S reduction} and \eqref{eqn: S reduction}})
						\notag\\
				&=\sum_{k=k_0-1}^{\ell}\!\Omega_i^{(k)}+\sum_{m=1}^{k_0-2}R_m^{(k_0-1)}+\sum_{k=1}^{k_0-1}\sum_{m=1}^{k_0-2}S_{k,m}^{(k_0-1)}.
		\end{align}
		This completes the induction step.
	\end{proof}
	Having proved the claim, we apply \eqref{eqn: T induction formula} in the case $k_0=1$ to obtain $T_i=\sum_{k=1}^\ell\Omega_i^{(k)}$.  This proves the lemma.
\end{proof}

\appendix
\section{\texorpdfstring{Formulation of $\Omega_\mathrm{Higgs}$}{Formulation of Omega}}\label[appendix]{subsec: Omega independence}

Recall that
\begin{equation}
	\Omega_\mathrm{Higgs}((\dot A_1,\dot\Phi_1),(\dot A_2,\dot\Phi_2))=-\mathbbm i\int_C\tr\left(\dot A_1\wedge\dot\Phi_2-\dot A_2\wedge\dot\Phi_1\right)
\end{equation}
where $(\dot A_j,\dot\Phi_j)$ are deformations of a stable Higgs bundle $(\delbar_A,\Phi)\in M_{(Q,\vec r,\mathcal F)}(\sigma,\alpha)$ infinitesimally preserving the level sets $\mu_\C=0$ and $\mathrm{Lev}=0$ defining $M_{(Q,\vec r,\mathcal F)}(\sigma)$.  In this appendix we show that it is not necessary to assume $(\dot A_j,\dot\Phi)$ are orthogonal to gauge orbits (i.e. $(\dot A_j,\dot\Phi)$ need not be harmonic representatives of the deformation).  Indeed, defining this would require fixing a hermitian metric on $E$, whereas we would prefer to work purely algebraically.

Note that any tangent vector to the gauge orbit in $M_{(Q,\vec r,\mathcal F)}(\sigma,\alpha)$ is of the form
		\[\dot\gamma\cdot(\delbar_A,\Phi)=\frac{\de}{\de t}\Big|_{t=0}e^{t\dot\gamma}\cdot(\delbar_A,\Phi)=(\dot A_j-\delbar_A\dot\gamma,\dot\Phi+[\dot\gamma,\Phi])\]
	for $\dot\gamma\in\mathrm{Lie}(\mathfrak G_\C)$.
\begin{proposition}
	  Let $\dot\gamma_1,\dot\gamma_2\in\mathrm{Lie}(\mathfrak G_\C)$, and let $(\dot A_j',\dot\Phi_j')=(\dot A_j,\dot \Phi_j)+\dot\gamma_j\cdot(\delbar_A,\Phi)$ for $j=1,2$.  Then $\Omega_\mathrm{Higgs}((\dot A_1',\dot\Phi_1'),(\dot A_2',\dot\Phi_2'))=\Omega_\mathrm{Higgs}((\dot A_1,\dot\Phi_1),(\dot A_2,\dot\Phi_2))$.
\end{proposition}
\begin{proof}
	We follow \cite[Remark 6.1]{FMSW26}.  Since the deformations preserve $\mu_\C=0$ (recall $\mu_\C(\delbar_,\Phi)=\delbar_A\Phi$), we have
	\begin{equation}\label{eqn: Higgs de mu C}
		\delbar_A\dot\Phi_j+[\dot A_j,\Phi]=0
	\end{equation}
	for $j=1,2$.  Also, since the deformations preserve $\mathrm{Lev}=0$, the Higgs field deformations $\dot\Phi_j$ are strongly parabolic (i.e.\ Levi part is 0).  For shorthand write $\Omega=\Omega_\mathrm{Higgs}((\dot A_1,\dot\Phi_1),(\dot A_2,\dot\Phi_2))$ and $\Omega'=\Omega_\mathrm{Higgs}((\dot A_1',\dot\Phi_1'),(\dot A_2',\dot\Phi_2'))$.  Expanding terms, using the product rule, and then applying \eqref{eqn: Higgs de mu C},
	{\allowdisplaybreaks
	\begin{align}
		\Omega'
			&=\Omega_\mathrm{Higgs}((\dot A_1-\delbar_A\dot\gamma_1,\dot\Phi_1+[\dot\gamma_1,\Phi]),(\dot A_2-\delbar_A\dot\gamma_2,\dot\Phi_2+[\dot\gamma_2,\Phi]))
				\notag\\
			&=-\mathbbm i\int_C\tr\left((\dot A_1-\delbar_A\dot\gamma_1)\wedge(\dot\Phi_2+[\dot\gamma_2,\Phi])-(\dot A_2-\delbar_A\dot\gamma_2)\wedge(\dot\Phi_1+[\dot\gamma_1,\Phi])\right)
				\notag\\
			&=\Omega-\mathbbm i\int_C\tr\Bigg(\dot A_1\wedge[\dot\gamma_2,\Phi]-\dot A_2\wedge[\dot\gamma_1,\Phi]
					\notag\\
				&\qquad\qquad\qquad-\delbar_A\dot\gamma_1\wedge\dot\Phi_2
				+\delbar_A\dot\gamma_2\wedge\dot\Phi_1
					\notag\\
				&\qquad\qquad\qquad-\delbar_A\dot\gamma_1\wedge[\dot\gamma_2,\Phi]+\delbar_A\dot\gamma_2\wedge[\dot\gamma_1,\Phi]\Bigg)
					\notag\\
			&=\Omega-\mathbbm i\int_C\tr\Bigg(-\dot\gamma_2\wedge[\dot A_1,\Phi]+\dot\gamma_1\wedge[\dot A_2,\Phi]
					\notag\\
				&\qquad\qquad\qquad+\delbar_A\left(-\dot\gamma_1\wedge\dot\Phi_2+\dot\gamma_2\wedge\dot\Phi_1\right) +\dot\gamma_1\delbar_A\dot\Phi_2-\dot\gamma_2\delbar_A\dot\Phi_1
					\notag\\
				&\qquad\qquad\qquad-\delbar_A\dot\gamma_1\wedge[\dot\gamma_2,\Phi]+\delbar_A\dot\gamma_2\wedge[\dot\gamma_1,\Phi]\Bigg)
					\notag\\
			&=\Omega-\mathbbm i\int_C\tr\left(\delbar_A\left(-\dot\gamma_1\wedge\dot\Phi_2+\dot\gamma_2\wedge\dot\Phi_1-[\dot\gamma_1,\dot\gamma_2]\Phi\right)\right)
					\notag\\
			&=\Omega-\mathbbm i\sum_{i=1}^n\tr\left(-\dot\gamma_1(p_i)\res_{p_i}\dot\Phi_2+\dot\gamma_2(p_i)\res_{p_i}\dot\Phi_2-[\dot\gamma_1(p_i),\dot\gamma_2(p_i)]\res_{p_i}\Phi\right)
	\end{align}
	}
	Since $\dot\gamma_j(p_i)$ is parabolic and $\res_{p_i}\dot\Phi_j$ is nilpotent, the first two parts vanish under trace.  Decompose $\dot\gamma_j(p_i)=B_{j,i}^\mathrm{Levi}+B_{j,i}^\mathrm{nil}$ into Levi and nilpotent parts, and similarly decompose $\res_{p_i}\Phi=C_{i}^\mathrm{Levi}+C_{i}^\mathrm{nil}$.  Only the Levi parts pair nontrivially, so
		\[\tr\left([\dot\gamma_1,\dot\gamma_2]\res_{p_i}\Phi\right)=\tr\left([B_{1,i}^\mathrm{Levi},B_{2,i}^\mathrm{Levi}]C_{i}^\mathrm{Levi}\right)
		\]
	since $C_{i}^\mathrm{Levi}\in Z(\mathfrak l_i)$.  Therefore $\Omega'=\Omega$.
\end{proof}
\newpage
\bibliography{Main}
\bibliographystyle{math}

\end{document}